\documentclass[11pt,reqno]{amsart}

\usepackage{amscd,amsmath,amssymb,amsthm,amsfonts,dsfont}
\usepackage{mathtools,mathrsfs,yfonts}
\usepackage[all]{xy}

\usepackage{xcolor}
\usepackage[hypertexnames=true, citecolor = green, colorlinks = false, linkcolor = red, linktoc = none, pdffitwindow = false, urlbordercolor = white]{hyperref}%

\usepackage{amsrefs}

\usepackage{fullpage}

\numberwithin{equation}{section}

\newtheorem{thm}{Theorem}[section]
\newtheorem{corollary}[thm]{Corollary}
\newtheorem{lemma}[thm]{Lemma}

\newtheorem{proposition}[thm]{Proposition}

\theoremstyle{definition}

\theoremstyle{remark}
\newtheorem{remark}[thm]{Remark}

\newtheorem{example}[thm]{Example}

\newcommand\bp{\begin{proof}}
\newcommand\ep{\end{proof}}

\newcommand{\un}{\mathds{1}}

\renewcommand\ker{\operatorname{ker}}
\newcommand\Irr{\operatorname{Irr}}

\newcommand\fix{\mathrm{fix}}
\newcommand\triv{\mathrm{triv}}

\newcommand\R{\mathbb{R}}
\newcommand\N{\mathbb{N}}
\newcommand\Z{\mathbb{Z}}
\newcommand\T{\mathbb{T}}

\newcommand{\CG}[0]{\mathcal{G}} 
\newcommand{\CO}[0]{\mathcal{O}} 
\newcommand\LL{\mathcal L}

\newcommand\ssubset{\subset\hspace*{-1.7mm}\subset}

\begin{document}

\title[Equilibrium states on right LCM semigroup C$^*$-algebras]{The groupoid approach to equilibrium states on right LCM semigroup C$^*$-algebras}

\author[S.~Neshveyev]{Sergey Neshveyev}
\address{Department of Mathematics \\ University of Oslo \\ P.O. Box 1053 \\ Blindern \\ NO-0316 Oslo \\ Norway}
\email{sergeyn@math.uio.no, n.stammeier@gmail.com}

\author[N.~Stammeier]{Nicolai Stammeier}

\thanks{This research was supported by the Research Council of Norway (RCN FRIPRO 240362).} 

\begin{abstract}
Given a right LCM semigroup $S$ and a homomorphism $N\colon S\to[1,+\infty)$, we use the groupoid approach to study the KMS$_\beta$-states on $C^*(S)$ with respect to the dynamics induced by~$N$. We establish necessary and sufficient conditions for the existence and uniqueness of KMS$_\beta$-states. As an application, we show that the sufficient condition for the uniqueness obtained for so-called generalized scales is necessary as well. Our most complete results are obtained for inverse temperatures $\beta$ at which the $\zeta$-function of $N$ is finite. In this case we get an explicit bijective correspondence between the KMS$_\beta$-states on $C^*(S)$ and the tracial states on $C^*(\ker N)$.
\end{abstract}

\date{December 6, 2019; new version: January 29, 2020; minor changes: June 7, 2021}

\maketitle

\section{Introduction}

In the recent years analysis of equilibrium states on semigroup C$^*$-algebras $C^*(S)$ with respect to the dynamics $\sigma^N$ defined by a homomorphism $N\colon S\to[1,+\infty)$ has been an increasingly popular topic. Most of the examples studied in the literature actually arise from a particular class of semigroups - the quasi-lattice ordered ones and, a bit more generally, the right LCM semigroups, which are left cancellative semigroups such that the intersection of any two right principal ideals is either empty or a principal ideal again. Such pairs $(C^*(S),\sigma^N)$ provide a rich yet tractable class of noncommutative dynamical systems. Indeed, under the mild additional assumption that $1$ is an isolated point of $N(S)\subset[1,+\infty)$, any $\sigma^N$-KMS$_\beta$-state on $C^*(S)$ is completely determined by its restriction to $C^*(\ker N)$, see Lemma~\ref{lem:ker-reduction}. So if we have a good understanding of the trace state space of $C^*(\ker N)$, we can quickly obtain explicit formulas for all hypothetical KMS-states. This reduces the task to determining which formulas define states on $C^*(S)$.

This strategy has been successfully employed for the $ax+b$ semigroup $S=\Z_+\rtimes\N$ (and the obvious homomorphism $N\colon S\to\N$) in~\cite{LR}, which inspired a lot of subsequent work, semigroups arising from integer dilation matrices~\cite{LRR} and self-similar group actions~\cite{LRRW},  Baumslag--Solitar semigroups~\citelist{\cite{CaHR}\cite{BLRS}},  semigroups defined by algebraic dynamical systems~\cite{ABLS}, and arbitrary quasi-lattice ordered semigroups (and a class of homomorphisms $N$)~\cite{BLRSi}.

An attempt to distill the essential properties of the pair $(S,N)$ that make the analysis in these papers possible has been made in~\cite{ABLS} and~\cite{BLRS} and led to the notion of a \emph{generalized scale} $N$ on $S$, see Example~\ref{ex:generalized_scale} below for the precise definition. On the one hand, the resulting theory is quite satisfactory, as it covers a majority of examples studied in the literature. On the other hand, the standing assumptions seem excessively strong and exclude already such simple examples as free abelian monoids of rank $\ge2$ with nontrivial homomorphisms~$N$. It is therefore natural to look for a more flexible approach.

In this work we will use the groupoid picture of $C^*(S)$ for right LCM semigroups $S$ to study the equilibrium states on $C^*(S)$ with respect to $\sigma^N$. Our goal is twofold - one, to see how far this approach can lead us with as little assumptions as possible on $(S,N)$, and two, to check whether it is possible to formulate the results we get in combinatorial/semigroup-theoretic terms not involving some obscure assumptions on the groupoid $\CG(S)$ underlying $C^*(S)$.

\smallskip

In detail, the contents of this paper is as follows. In the preliminary Section~\ref{sec:groupoid} we recall the groupoid picture for $C^*(S)$ as presented in \citelist{\cite{Nor}\cite{Pat}}. In addition, we include the description of KMS-states on the groupoid C$^*$-algebras in terms of quasi-invariant measures and fields of tracial states on the C$^*$-algebras of the isotropy groups given in~\cite{Nes}. In this context, we briefly explain that although the second countability assumption in~\cite{Nes} forces us to consider countable~$S$, all our main results can be extended to general right LCM semigroups by passing to a limit, once they are formulated in a groupoid-free form.

\smallskip

In Section~\ref{sec:measures} we begin by proving that
$$
s(\ker N)\cap tS\ne\emptyset \quad \text{for all } s \in S \text{ and } t\in \ker N
$$
is a necessary condition for the existence of a $\sigma^N$-KMS$_\beta$-state on $C^*(S)$ for some $\beta>0$. Once this condition is satisfied, we get a quasi-lattice ordered set $S/{\sim_N}$. This set plays a central role in all subsequent considerations. In particular, Proposition~\ref{prop:muN} shows that the existence of a KMS$_\beta$-state, or in other words, the existence of a quasi-invariant probability measure $\mu_{N,\beta}$ on the unit space~$\CG(S)^{(0)}$ of the groupoid~$\CG(S)$ with a prescribed Radon--Nikodym cocycle, is equivalent to a (generally infinite) system of inequalities for the map $S/{\sim_N}\to[1,+\infty)$ induced by $N^\beta$, cf.~\cite{ALN}*{Section~2}.

\smallskip

In Section~\ref{sec:KMS} we consider $\beta$ such that
$$
\zeta_N(\beta):=\sum_{[s]\in S/{\sim_N}}N(s)^{-\beta}<+\infty.
$$
In this case we have a complete description of the KMS$_\beta$-states: by Theorem~\ref{thm:finiteKMS}, they are in a bijective correspondence with the tracial states on $C^*(\ker N)$. Conceptually the picture is very simple (cf.~\cite{Nes}*{Section~3}): the measure $\mu_{N,\beta}$ is concentrated on one orbit of the action of $\CG(S)$ on $\CG(S)^{(0)}$, so the fields of tracial states corresponding to the KMS$_\beta$-states are parameterized by the tracial states on the group C$^*$-algebra of the isotropy group of any chosen point on the orbit. Picking a particular such point we get the enveloping group of $\ker N$ as the isotropy group, so we can as well parameterize the KMS$_\beta$-states by the tracial states on $C^*(\ker N)$. Furthermore, the orbit can be canonically identified with~$S/{\sim_N}$, which allows us to write down the correspondence between the KMS$_\beta$-states and the traces explicitly, without appealing to the groupoid $\CG(S)$. This result confirms the general principle that once $\beta$ is large enough, so that an appropriate $\zeta$-function is finite, all the KMS$_\beta$-states can be explicitly described as Gibbs-like states, cf.~\cite{ALN}*{Section~7}. We also obtain corresponding results for the KMS$_\infty$ and ground states in this section.

\smallskip

In Section~\ref{sec:unique} we fix $\beta$ which admits a KMS$_\beta$-state in order to study the question when this state is unique. By a general result on groupoid C$^*$-algebras, uniqueness holds if and only if the isotropy bundle  of $\CG(S)$ is essentially trivial with respect to the measure $\mu_{N,\beta}$. We show that it is possible to write this condition in a combinatorial form, see Theorem~\ref{thm:unique}. Although the formulas and the uniqueness criterion that we get become complicated, we show in Example~\ref{ex:generalized_scale} that for generalized scales our criterion reduces to the conditions introduced in~\cite{BLRS} and therefore covers most examples in the literature. As a consequence, the conditions in~\cite{BLRS} are not only sufficient for the uniqueness, but also necessary. We finish the section with a short comment on simplicity of the boundary quotient of~$C^*(S)$, see Proposition~\ref{prop:BQ simple}.

\smallskip

Finally, let us say a few words about what is \emph{not} done in this paper. The remaining problem is to be able to say something meaningful in the case when $\zeta_N(\beta)=+\infty$, a KMS$_\beta$-state exists, yet the uniqueness criterion is not satisfied. In this case the isotropy bundle is not essentially trivial with respect to $\mu_{N,\beta}$ and one should expect that the measure $\mu_{N,\beta}$ is nonatomic. There are few general results describing all possible fields of tracial states defining the KMS-states in such situations, see~\citelist{\cite{Nes}\cite{Ch}}. It therefore remains to be seen how useful the groupoid approach to the semigroup C$^*$-algebras is in this case.

\smallskip

\paragraph{\bf Acknowledgement} {This work would probably not exist today if it had not been for the stimulating atmosphere during the \emph{Cuntz--Pimsner Cross-Pollination} workshop at the Lorentz Center in Leiden in June 2018, which is why the authors would like to thank the organizers for providing an opportunity away from the regular work group encounters.}

\section{The groupoid picture}\label{sec:groupoid}

Let $S$ be a right LCM semigroup, by which we mean a left cancellative semigroup with identity element~$e$ such that for all $s,t\in S$ we have either $sS\cap tS=\emptyset$ or $sS\cap tS=rS$ for some $r\in S$. If $t\in sS$ for some $s,t\in S$, we denote by $s^{-1}t$ the unique element of $S$ such that $s(s^{-1}t)=t$.

The semigroup C$^*$-algebra $C^*(S)$ is generated by the isometries $v_s$, $s\in S$, satisfying the relations
$$
v_e=1,\ \ v_sv_t=v_{st},\ \ \text{and}\ \ v_s^{\phantom{*}}v_s^*v_t^{\phantom{*}}v_t^*=\begin{cases}v_r^{\phantom{*}}v_r^*,&\text{if}\ \ sS\cap tS=rS,\\0,&\text{if}\ \ sS\cap tS=\emptyset.\end{cases}
$$
The elements $v_s^{\phantom{*}}v_t^*$, for all $s,t\in S$, span a dense subspace of $C^*(S)$.

By a \emph{scale} on $S$ we mean a semigroup homomorphism $N\colon S\to([1,+\infty),\cdot)$. Every scale $N$ defines a one-parameter group of automorphisms of $C^*(S)$ by
$$
\sigma^N_t(v_s)=N(s)^{it}v_s.
$$
We are interested in understanding the $\sigma^N$-KMS-states on $C^*(S)$. To be precise, for $\beta\in\R$, by a $\sigma^N$-KMS$_\beta$-state we mean a state $\varphi$ on $C^*(S)$ such that
\[
\varphi(ab) = \varphi(b \sigma^N_{i\beta}(a))
\]
for all $a$ and $b$ in a dense subspace of $\sigma^N$-analytic elements of $C^*(S)$.

Our approach relies on the groupoid picture of $C^*(S)$. For its presentation, let us first consider the left inverse hull $I_\ell(S)$ of $S$, that is, the inverse semigroup $\{\lambda_s^{\phantom{1}}\lambda_t^{-1} \mid s,t \in S\}\cup\{0\}$ given by the partial bijections $\lambda_s\colon S \to sS$, $t\mapsto st$. By \cite{Nor}*{Proposition~3.3.2}, we have a canonical isomorphism $C^*(S)\cong C^*_0(I_\ell(S))$, where the subscript $0$ indicates that we consider only those representations of $I_\ell(S)$ by partial isometries that kill $0\in I_\ell(S)$. The C$^*$-algebra $C^*_0(I_\ell(S))$ has a well-known description as a groupoid C$^*$-algebra, see~\cite{CELY}*{Section~5.5} or \cite{Pat}*{Chapter~4}. We shall now recall the construction of the corresponding groupoid $\CG(S)$, as its structure is central to the subsequent parts of this work.

Let $E(S):=\{e_s:= \lambda_s^{\phantom{1}}\lambda_s^{-1} \mid s \in S\}\cup\{0\}$ be the semilattice of idempotents of $I_\ell(S)$ and
$$
\widehat{E(S)}:= \{\chi\colon E(S) \to \{0,1\} \mid \chi \text{ is a nonzero semigroup homomorphism},\ \chi(0)=0\}
$$
be the space of characters on $E(S)$. Define
$$
\Sigma := \{(\lambda_s^{\phantom{1}}\lambda_t^{-1},\chi) \in I_\ell(S)\times\widehat{E(S)} \mid \chi(e_t)=1\},
$$
as well as an equivalence relation $\sim$ on $\Sigma$ such that $(\lambda_{s_1}^{\phantom{1}}\lambda_{t_1}^{-1},\chi_1) \sim (\lambda_{s_2}^{\phantom{1}}\lambda_{t_2}^{-1},\chi_2)$ iff
\[ \chi_1=\chi_2 \text{ and } \exists\, r \in S: \lambda_{s_1}^{\phantom{1}}\lambda_{t_1}^{-1}e_r=\lambda_{s_2}^{\phantom{1}}\lambda_{t_2}^{-1}e_r \text{ and } \chi_1(e_r)=1.\]

This equivalence relation can be characterized as follows.

\begin{lemma}\label{lem:equivalence relation explained}
We have $(\lambda_{s_1}^{\phantom{1}}\lambda_{t_1}^{-1},\chi)\sim(\lambda_{s_2}^{\phantom{1}}\lambda_{t_2}^{-1},\chi)$ if and only if there exist $r_1,r_2 \in S$ with
\[s_1r_1=s_2r_2,\ \ t_1r_1=t_2r_2\ \ \text{and}\ \ \chi(e_{t_1r_1}) =1.\]
\end{lemma}

\begin{proof}
Assume $(\lambda_{s_1}^{\phantom{1}}\lambda_{t_1}^{-1},\chi) \sim (\lambda_{s_2}^{\phantom{1}}\lambda_{t_2}^{-1},\chi)$, that is, there is $r\in S$ with $\lambda_{s_1}^{\phantom{1}}\lambda_{t_1}^{-1}e_r = \lambda_{s_2}^{\phantom{1}}\lambda_{t_2}^{-1}e_r$ and $\chi(e_r)=1$.
In particular, this implies that the corresponding domain idempotents match: $e_{t_1}e_r=e_{t_2}e_r$.
Hence $t_1S\cap rS=r'S=t_2S\cap rS$ for some $r'\in S$, and then $\chi(e_{r'})=1$, as $\chi$ is nonzero on~$e_{t_i}$ and~$e_r$. We have $t_1r_1=r'=t_2r_2$ for some $r_1$ and $r_2$.
Then
\[ \lambda_{s_1r_1}^{\phantom{1}}\lambda_{t_1r_1}^{-1} =\lambda_{s_1}^{\phantom{1}}\lambda_{t_1}^{-1}e_r = \lambda_{s_2}^{\phantom{1}}\lambda_{t_2}^{-1}e_r = \lambda_{s_2r_2}^{\phantom{1}}\lambda_{t_2r_2}^{-1}, \]
which forces $s_1r_1=s_2r_2$. This proves the lemma in one direction.

Conversely, if $s_1r_1=s_2r_2$, $t_1r_1=t_2r_2$ and $\chi(e_{t_1r_1}) =1$, then letting $r=t_1r_1=t_2r_2$, we easily see that $\lambda_{s_1}^{\phantom{1}}\lambda_{t_1}^{-1}e_r=\lambda_{s_2}^{\phantom{1}}\lambda_{t_2}^{-1}e_r$ and  $\chi(e_r)=1$.
\end{proof}

As a set, our groupoid is
$$
\CG(S):= \Sigma/{\sim},
$$
and we denote by $[x,\chi]\in\CG(S)$ the class of $(x,\chi)\in\Sigma$.

We consider $\widehat{E(S)}$ as a subset of $\CG(S)$ by identifying $\chi\in\widehat{E(S)}$ with $[e_t,\chi]$, where $t\in S$ is any element (for example, the identity element) such that $\chi(e_t)=1$. This makes sense, since if $s$ and $t$ are two such elements, then $\chi(e_se_t)=1$, hence $sS\cap tS=rS$ for some $r$ such that $\chi(e_r)=1$, and therefore $[e_s,\chi]=[e_t,\chi]$.

For $(x,\chi) \in \Sigma$, define $x\cdot\chi$ by $$(x\cdot\chi)(e_r) := \chi(x^{-1}e_rx).$$ Two elements $[x,\chi]$ and $[y,\psi]$ in $\CG(S)$ are composable if $\chi=y\cdot \psi$, in which case we set $$[x,\chi]\cdot [y,\psi] := [xy,\psi].$$
Thus, the unit space of $\CG(S)$ is $\widehat{E(S)}$, and the source and range maps are given by
$$
s([x,\chi]) = \chi,\ \ r([x,\chi]) = x\cdot\chi.
$$
The inversion is given by $$[x,\chi]^{-1} = [x^{-1},x\cdot\chi].$$

The set $\widehat{E(S)}$ is a compact space in the topology of pointwise convergence.
We then define a topology on $\CG(S)$ by taking as a basis of topology
the sets $$D(\lambda_s^{\phantom{1}}\lambda_t^{-1},U):= \{[\lambda_s^{\phantom{1}}\lambda_t^{-1},\chi] \mid \chi \in U\},$$
where $s,t \in S$ and $U$ is an open subset of
$$
Z_t:=\{\chi \in \widehat{E(S)} \mid \chi(e_t)=1\}.
$$
One checks that this turns $\CG(S)$ into a locally compact, but not necessarily Hausdorff, \'{e}tale groupoid. To be precise, by a locally compact \'{e}tale groupoid we mean a groupoid $\CG$ endowed with a topology such that:
\begin{enumerate}
\item[-] the groupoid operations are continuous;
\item[-] every point of $\CG$ has a compact Hausdorff neighbourhood;
\item[-] the unit space $\CG^{(0)}$ and the fibers $\CG^\chi=r^{-1}(\chi)$ of the range map $r$ are Hausdorff;
\item[-] the map $r$ is a local homeomorphism;
\end{enumerate}
see~\cite{Pat}*{Definitions~2.2.1 and~2.2.3} or~\cite{MW}*{p.6}. The groupoid $\CG(S)$ is second countable if $S$ is countable.

Given a locally compact \'{e}tale groupoid $\CG$, the space $C_c(\CG)$ is defined as the linear span of functions $f$ such that $f$ is continuous, with compact support, on an open Hausdorff set and zero outside that set. Note that such functions $f$ are not necessarily continuous on $\CG$ (and for this reason the space $C_c(\CG)$ is denoted by $\mathscr{C}(\CG)$ in~\cite{MW}). Being equipped with the convolution product and the involution $f^*(g)=\overline{f(g^{-1})}$, the space $C_c(\CG)$ becomes a $*$-algebra, and its C$^*$-enveloping algebra is denoted by $C^*(\CG)$.

For any right LCM semigroup $S$, we have an isomorphism $C^*(S)\cong C^*(\CG(S))$ given by
\begin{equation} \label{eq:groupoid-iso}
v_s^{\phantom{*}}v_t^* \mapsto \un_{D(\lambda_s^{\phantom{1}}\lambda_t^{-1},Z_t)},
\end{equation}
where $\un_X$ denotes the characteristic function of a set $X$.

\smallskip

The following well-known lemma gives an alternative description of the subalgebra $C(\widehat{E(S)})\subset C^*(\CG(S))$ and will be used repeatedly in the sequel.

\begin{lemma}\label{lem:lattice-set}
We have a C$^*$-algebra isomorphism of $C(\widehat{E(S)})$ onto the norm closure of $$\operatorname{span}\{\un_{sS}\mid s\in S\}\subset\ell^\infty(S)$$ that maps $\un_{Z_s}$ into $\un_{sS}$ for all $s\in S$.
\end{lemma}

In particular, it follows that $Z_s\cap Z_t=Z_r$ if $sS\cap tS=rS$ for some $r$, and $Z_s\cap Z_t=\emptyset$ if $sS\cap tS=\emptyset$, which is also easy to see by definition.
Hence as a basis of topology on $\widehat{E(S)}$ we can take the cylinder sets
$$
Z_{s,F} := \{\chi \in \widehat{E(S)} \mid \chi(e_s)=1,\ \chi(e_t)=0 \text{ for all } t\in F\}=Z_s\setminus\bigcup_{t\in F}Z_t,
$$
where $s \in S$ and $F$ is a finite subset of $S$, which we will sometimes write as $F\ssubset S$. Furthermore, it is enough to consider $F\ssubset sS$.

\smallskip

Now, every scale $N\colon S \to [1,+\infty)$ gives rise to an $\R$-valued $1$-cocycle on $\CG(S)$ defined by
$$
c_N([\lambda_r^{\phantom{1}}\lambda_s^{-1},\chi]):=\log N(r)-\log N(s),
$$
which induces $\sigma^N$ on $C^*(S)\cong C^*(\CG(S))$ in the sense that
$$
\sigma^N_t(f)(g)=e^{itc_N(g)}f(g)\ \ \text{for}\ \ f\in C_c(\CG(S)),\ g\in\CG(S).
$$

For such dynamics we have the following description of KMS-states.

\begin{thm}\label{thm:KMS}
Suppose that  $c$ is a continuous $\R$-valued $1$-cocycle on a locally compact, not necessarily Hausdorff, second countable \'{e}tale groupoid $\CG$. Let $\sigma^c$ be the corresponding dynamics on~$C^*(\CG)$. Then, for every $\beta\in \R$, there exists a one-to-one correspondence between the $\sigma^c$-KMS$_\beta$-states on~$C^*(\CG)$ and the pairs $(\mu,\{\varphi_x\}_{x\in \CG^{(0)}})$ consisting of a probability measure $\mu$ on $\CG^{(0)}$ and a $\mu$-measurable field of tracial states~$\tau_x$ on $C^*(\CG^x_x)$ such that:
\begin{enumerate}
\item[(i)] $\mu$ is quasi-invariant with Radon--Nikodym cocycle $e^{-\beta c}$;
\item[(ii)] $\tau_x(u_g)=\tau_{r(h)}(u_{hgh^{-1}})$ for $\mu$-a.e.~$x$ and all $g\in \CG^x_x$ and $h\in \CG_x$.
\end{enumerate}
Namely, the state corresponding to $(\mu,\{\tau_x\}_x)$ is given~by
$$
\varphi(f)=\int_{\CG^{(0)}}\sum_{g\in \CG^x_x}f(g)\tau_x(u_g)d\mu(x)\ \ \hbox{for}\ \ f\in C_c(\CG).
$$
\end{thm}

\bp
In the Hausdorff case this is~\cite{Nes}*{Theorem~1.3}. The result remains true in the non-Hausdorff case, with essentially the same proof, since all the computations in the proof need only functions that are nonzero on a compact Hausdorff subset, and since the main technical result used in the proof - Renault's disintegration theorem - is valid in the non-Hausdorff case as well, see~\cite{Ren2} or~\cite{MW}*{Theorem~7.8}.
\ep

Once we have a quasi-invariant probability measure $\mu$ on $\CG^{(0)}$ with Radon--Nikodym cocycle $e^{-\beta c}$, we have two natural choices for measurable fields of tracial states. One is to take the canonical trace on $C^*(\CG^x_x)$ for every $x$. We denote the corresponding KMS$_\beta$-state by $\varphi_\mu'$. The other possibility is to take the trivial character on $C^*(\CG^x_x)$ for every $x$. We denote the corresponding KMS$_\beta$-state by $\varphi_\mu''$. Thus, for $f\in C_c(\CG)$,
\begin{equation}\label{eq:2states 1}
\varphi'_\mu(f)=\int_{\CG^{(0)}}f(x)d\mu(x),
\end{equation}
while
\begin{equation}\label{eq:2states 2}
\varphi''_\mu(f)=\int_{\CG^{(0)}}\sum_{g\in \CG^x_x}f(g)d\mu(x).
\end{equation}
It is clear that these two states coincide if and only if the $\mu$-measure of points with nontrivial isotropy groups $\CG^x_x$ is zero. Hence we get the following result.

\begin{corollary}\label{cor:uniqueKMS}
In the setting of Theorem~\ref{thm:KMS}, a quasi-invariant probability measure $\mu$ on $\CG^{(0)}$ with Radon--Nikodym cocycle $e^{-\beta c}$ corresponds to a unique $\sigma^c$-KMS$_\beta$-state if and only if $\mu$-almost all points have trivial isotropy.

In particular, $C^*(\CG)$ has a unique $\sigma^c$-KMS$_\beta$-state if and only if there is a unique quasi-invariant probability measure $\mu$ on $\CG^{(0)}$ with Radon--Nikodym cocycle $e^{-\beta c}$ and $\mu$-almost all points have trivial isotropy.
\end{corollary}

\begin{remark}
When $\CG$ is Hausdorff, the state $\varphi'_\mu$ is the composition of the state on $C_0(\CG^{(0)})$ defined by $\mu$ with the canonical conditional expectation $C^*(\CG)\to C_0(\CG^{(0)})$, $C_c(\CG)\ni f\mapsto f|_{\CG^{(0)}}$. When $\CG$ is non-Hausdorff, this conditional expectation may not exist, since $f|_{\CG^{(0)}}$ is not necessarily continuous for $f\in C_c(\CG)$.
\end{remark}

For the groupoid $\CG(S)$ associated with a right LCM semigroup $S$ and the cocycle $c_N$ defined by a scale $N$ on $S$, condition (i) in Theorem~\ref{thm:KMS} means simply that
\begin{equation}\label{eq:scaling}
\mu(\lambda_s\cdot A)=N(s)^{-\beta}\mu(A)
\end{equation}
for all Borel subsets $A\subset \CG(S)^{(0)}=\widehat{E(S)}$ and $s\in S$.

\smallskip

We will start analyzing condition~\eqref{eq:scaling} in the next section. We want to finish this section by pointing out that although Theorem~\ref{thm:KMS} requires second countability, all our main results on the KMS-states on $C^*(S)$ have formulations independent of the groupoid picture and can be generalized to uncountable right LCM semigroups. We leave details of such a generalization to the interested reader and confine ourselves to outlining a general strategy.

A submonoid $T$ of a right LCM semigroup $S$ that is a right LCM semigroup satisfying
\begin{equation*}\label{eq:ideals in submonoid}
(sT\cap tT)S = sS \cap tS\ \ \text{for all}\ \ s,t \in T
\end{equation*}
will be called a \emph{right LCM submonoid} (of $S$).
We then have a canonical homomorphism $C^*(T)\to C^*(S)$. In general this homomorphism may not be injective. But we will show in Corollary~\ref{cor:honest subalgebra} that this canonical map is injective if the right LCM submonoid $T\subset S$ is also \emph{hereditary}, that is, whenever $sS$ contains an element of $T$ for some $s\in S$, we must have $s\in T$.

\begin{remark}\label{rem:hereditary}
If $T\subset S$ is a  hereditary LCM submonoid, then $S^*\subset T$, and whenever $st\in T$ for some $s,t\in S$, we must have $s,t\in T$. The first property is immediate, as $e\in T$ and $T$ is hereditary. Next, if $st\in T$, then $s\in T$. But then the equality $stS\cap sS=stS$ implies that $stT\cap sT=stuT$ for some $u\in S^*$, since $T$ is a right LCM submonoid of $S$. It follows that $st\in sT$, hence $t\in T$ by left cancellation.
\end{remark}

For any right LCM semigroup $S$ we can construct an increasing net $(S_i)_i$ of countable right LCM submonoids with union $S$ as follows. Take a countable submonoid $X\subset S$. For every pair of elements $s,t\in X$ such that $sS\cap tS\ne\emptyset$ pick elements $a_{s,t},b_{s,t}\in S$ such that $sa_{s,t}=tb_{s,t}$ and $sS\cap tS=sa_{s,t}S$. Then, for each $c\in sX\cap tX$, we have $c=sa_{s,t}d_c$ for some $d_c\in S$. Consider the new submonoid generated by $X$ and the elements $a_{s,t}$, $b_{s,t}$ and $d_c$ for all $s,t\in X$ and $c\in sX\cap tX$. Repeat the same procedure for the new submonoid, but keep the same elements $a_{s,t}$ and $b_{s,t}$ whenever they are already defined. By repeating this procedure we get an increasing sequence of countable submonoids of $S$. Denote by $S_X$ its union. Then $S_X$ is a countable right LCM submonoid of $S$. By varying $X$ we get the required net of submonoids.

Note that we might want to use a more elaborate procedure for constructing $S_X$. For example, in the next section we will consider an equivalence relation $\sim_N$ on $S$ arising from a scale $N$. In this case, we may throw in additional elements of $\ker N$ at every stage of our construction of $S_X$ to make sure that the equivalence relation on $S_X$ defined by $N|_{S_X}$ is the restriction of $\sim_N$.

Once a net $(S_i)_i$ is constructed, the C$^*$-algebra $C^*(S)$ becomes the inductive limit of $C^*(S_i)$, and this allows one to extend results from countable right LCM semigroups to arbitrary ones by passing to a limit. From now on we will therefore often assume that $S$ is countable.

\section{Scales and quasi-invariant measures}\label{sec:measures}

Let $S$ be a countable right LCM semigroup and $N\colon S\to[1,+\infty)$ be a scale. Our goal is to understand condition~\eqref{eq:scaling} in combinatorial terms.

\smallskip

We start by identifying simple necessary conditions for the existence of a quasi-invariant probability measure $\mu$ with Radon--Nikodym cocycle $e^{-\beta c_N}$. Denote by $\ker N$ the ``naive'' kernel $N^{-1}(1)$ of $N$.

\begin{lemma} \label{lem:neccessary}
For any nontrivial scale $N$ and $\beta\in\R$, necessary conditions for the existence of a probability measure $\mu$ satisfying~\eqref{eq:scaling} (and hence for the existence of a $\sigma^N$-KMS$_\beta$-state on~$C^*(S)$) are that $\beta=0$ or that $\beta>0$ and
\begin{equation}\label{eq:admis}
\text{for any}\ \ s\in S,\ t\in\ker N,\ \ \text{we have}\ \ sS\cap tS=rS\ \ \text{for some}\ r\in s(\ker N).
\end{equation}
\end{lemma}

\bp
For any $s\in S$, we have $\lambda_s\cdot Z_e=Z_s$, hence
$\mu(Z_s)=N(s)^{-\beta}$. As $N$ is assumed to be nontrivial and we must have $\mu(Z_s)\le1$ for all $s$, this implies that $\beta\ge0$.

Assume now that $\beta>0$ and take $s\in S$, $t\in\ker N$. Then $\mu(Z_t)=1$, hence
$$
\mu(Z_s\cap Z_t)=\mu(Z_s)=N(s)^{-\beta}.
$$
By the definition of the sets $Z_r$ or by Lemma~\ref{lem:lattice-set}, it follows that $sS\cap tS=rS$ for some $r$ such that $N(r)^{-\beta}=N(s)^{-\beta}$. Hence $r\in s(\ker N)$.
\ep

Recall from~\cites{Star,CrispLaca} that
$$
S_c:= \{a \in S\mid aS\cap sS\ne\emptyset\ \ \forall s \in S\}
$$
is called the \emph{core} subsemigroup of $S$. Condition~\eqref{eq:admis} implies that $\ker N$ is a right LCM submonoid of $S$ contained in $S_c$. This submonoid is clearly hereditary.

\begin{remark}\label{rem:admis}
A formally weaker, but nevertheless equivalent to~\eqref{eq:admis}, condition is that
\begin{equation}\label{eq:admis2}
s(\ker N)\cap tS\ne\emptyset\ \ \text{for all}\ \ s\in S\ \text{and}\ t\in\ker N.
\end{equation}
Indeed, if \eqref{eq:admis2} is satisfied, then $sS\cap tS=rS\ni sa$ for some $r\in S$ and $a\in\ker N$. Applying $N$ we get $N(s)\le N(r)\le N(sa)=N(s)$, hence $r\in s(\ker N)$, that is, \eqref{eq:admis} holds.
\end{remark}

The following lemma allows one to easily verify condition~\eqref{eq:admis} in a number of examples.

Let us first introduce some notation. Given a submonoid $T\subset S$, we have a partial order on the set $T/T^*$ defined by divisibility in $T$: $t_1T^*$ is larger than $t_2T^*$ iff $t_1T\subset t_2T$. If $T/T^*$ happens to be directed, or in other words, $t_1T\cap t_2T\ne\emptyset$ for any $t_1,t_2\in T$, then we can define an equivalence relation $\sim_T$ on $S$ by
\begin{equation}\label{eq:equiv}
s\sim_T t\ \ \text{iff}\ \ \exists\ a,b\in T:\ sa=tb.
\end{equation}
When $T=\ker N$ and $(\ker N)/(\ker N)^*$ is directed, we will write $\sim_N$ instead of $\sim_{\ker N}$ and denote by $[s]$ the equivalence class of $s\in S$.

\begin{lemma}\label{lem:finite-level}
Assume a scale $N$ on $S$ is such that $(\ker N)/(\ker N)^*$ is directed and the sets $N^{-1}(\lambda)/{\sim_N}$ are finite for all $\lambda\in N(S)$. Then $N$ satisfies condition~\eqref{eq:admis}.
\end{lemma}

\bp
As $S$ is left cancellative, it is easy to see that $S$ acts on $S/{\sim_N}$ by injective maps such that $s[t]=[st]$. Since $\ker N$ preserves the finite sets $N^{-1}(\lambda)/{\sim_N}$, it follows that $\ker N$ acts by bijective maps. In other words, for any $s\in S$ and $t\in\ker N$ we have $s\sim_N tr$ for some $r$. But this means that $s(\ker N)\cap tS\ne\emptyset$, which is equivalent to condition~\eqref{eq:admis} by Remark~\ref{rem:admis}.
\ep

\emph{From now on we assume that $N$ satisfies condition~\eqref{eq:admis}.}

\smallskip

Then $(\ker N)/(\ker N)^*$ is directed, so that $\sim_N$ is well-defined.

\begin{lemma}
Assume we have $sS\cap tS=rS$ for some $r,s,t\in S$. Then, for any $t'\sim_N t$, we have $sS\cap t'S=r'S$ for some $r'\sim_N r$.
\end{lemma}

\bp
By assumption we have $ta=t'b$ for some $a,b\in\ker N$. It suffices to prove the lemma for the pairs $(t,ta)$ and $(t',t'b)$ instead of $(t,t')$. Therefore we may assume that $t'=ta$.

We have $r=tx$ for some $x\in S$. Then by condition~\eqref{eq:admis} we have $xS\cap aS=x'S$ for some $x'\sim_N x$. Hence $tx'\sim_N tx=r$ and
$$
tx'S=txS\cap taS=sS\cap tS\cap taS=sS\cap t'S,
$$
so that we can take $r'=tx'$.
\ep

This lemma implies that we have a well-defined binary relation $\le$ on $S/{\sim_N}$ given by
\begin{equation}\label{eq:order}
[s]\le[t]\ \ \text{iff}\ \ sS\cap tS=t'S\ \ \text{for some}\ \ t'\sim_Nt,
\end{equation}
or equivalently, $[s]\le[t]$ iff $sS\supset t'S$ for some $t'\sim_N t$.

\begin{lemma}\label{lem:quasi-lattice}
The pair $(S/{\sim_N},\le)$ is a quasi-lattice, that is, a partially ordered set such that any two elements $[s]$ and $[t]$ either have a least upper bound $[s]\vee[t]$ or do not have a common upper bound at all, in which case we write $[s]\vee[t]=\infty$. Furthermore,
$$
[s]\vee[t]=\begin{cases}[r], &\text{if}\ \ sS\cap tS=rS,\\ \infty,&\text{if}\ \ sS\cap tS=\emptyset.\end{cases}
$$
\end{lemma}

\bp
The reflexivity is obvious. To show antisymmetry, assume $[s]\le[t]\le[s]$. Then $sS\supset t'S\supset s'S$ for some $s'\sim_N s$ and $t'\sim_N t$. We have $t'=sa$ and $s'=t'b$ for some $a,b\in S$, hence $s'=sab$. Applying $N$ we conclude that $a,b\in\ker N$, so that $t\sim_N t'\sim_N s$.

To show transitivity, assume $[r]\le[s]\le[t]$. By replacing first $s$ by an equivalent element to get $rS\supset sS$, and then~$t$ to get $sS\supset tS$, we may assume that $rS\supset sS\supset tS$, in which case we clearly have $[r]\le[t]$. Therefore $\le$ is a partial order on $S/{\sim_N}$.

Next, we have to show that if $[s]\le[x]$ and $[t]\le[x]$ for some $s,t,x\in S$, then $sS\cap tS =rS$ for some $r$ and $[r]\le[x]$. We have $sS\supset x'S$ and $tS\supset x''S$ for some $x'$ and $x''$ equivalent to~$x$. But then $x'S\cap x''S=x'''S$ for some $x'''\sim_N x$. Therefore by replacing $x$ by $x'''$ we may assume that $sS\cap tS\supset xS$, in which case the existence of $r$ becomes obvious.
\ep

\begin{lemma}\label{lem:action}
We have an action of $S$ on $S/{\sim_N}$ by injective order-preserving maps such that $s[t]=[st]$. The elements of $\ker N$ act by bijective maps.
\end{lemma}

\bp
The first statement is immediate (and was already observed in the proof of Lemma~\ref{lem:finite-level}). In order to prove the second one, take $a\in\ker N$ and $s\in S$. Then $aS\cap sS=s'S$ for some $s'\sim_N s$. Write $s'$ as $s'=at$. Then $a[t]=[s]$.
\ep

\begin{example}\label{ex:ax+b}
Consider the right LCM semigroup $S=\Z_+\rtimes\N$ studied in~\cite{LR}, so the elements of $S$ are pairs $(c,n)$ of integers such that $c\ge0$ and $n\ge1$, and the product is given by $(c,n)(d,m)=(c+nd,nm)$.

Define a scale $N$ by $N(c,n)=n$. As $\Z_+$ is directed, the equivalence relation $\sim_N$ is well-defined. It is easy to check that
$$
(c,n)\sim_N(d,m)\ \ \text{if and only if}\ \ n=m\ \text{and}\ \ c\equiv d\mod n.
$$
Therefore $N^{-1}(n)/{\sim_N}$ can be identified with $\Z/n\Z$, so $S/{\sim_N}$ is the disjoint union $\bigsqcup^\infty_{n=1}\Z/n\Z$.
It is then not difficult to show that if $[c]\in \Z/n\Z$ and $[d]\in\Z/m\Z$, then
$$
[c]\le[d]\ \ \text{if and only if}\ \ n|m\ \text{and}\ \ c\equiv d\mod n.
$$

The action of $S$ on $\bigsqcup^\infty_{n=1}\Z/n\Z$ is given as follows: an element $(c,n)$ maps $[d]\in\Z/m\Z$ into $[c+nd]\in\Z/nm\Z$.
\hfill$\diamondsuit$
\end{example}

Let us introduce the following notation. For a finite subset $K\subset S/{\sim_N}$ we put
$$
[q_K]:=\bigvee_{[s]\in K}[s]\in (S/{\sim_N})\cup\{\infty\},
$$
so $q_K$ denotes any representative of the class on the right, whenever this class is finite, and $q_K=\infty$ otherwise.
We then have the following result, cf.~\cite{ALN}*{Section~2}.

\begin{proposition}\label{prop:muN}
For any scale $N$ satisfying~\eqref{eq:admis} on a countable right LCM semigroup~$S$ and any $\beta\ge0$, a probability measure $\mu$ on $\widehat{E(S)}$ satisfying~\eqref{eq:scaling} is unique if it exists, and we have
\begin{equation}\label{eq:muN}
\mu(Z_{s,F})=N(s)^{-\beta}+\sum_{\emptyset\neq K\subset [F]} (-1)^{\vert K\vert} N(q_K)^{-\beta}
\end{equation}
for all $s \in S$ and $F\ssubset sS$, where $[F]$ denotes the image of $F$ in $S/{\sim_N}$ and we use the convention $N(\infty)^{-\beta}=0$.

Furthermore, such a measure exists (or, equivalently, a $\sigma^N$-KMS$_\beta$-state on~$C^*(S)$ exists) if and only if the expression on the right in~\eqref{eq:muN} is nonnegative for $s=e$ and all $[F]\ssubset \{[t]\in S/{\sim_N}: [e]<[t]\}$.
\end{proposition}

Whenever such a measure $\mu$ exists, we denote it by $\mu_{N,\beta}$. For $\beta=1$ we will write $\mu_N$ instead of~$\mu_{N,1}$.

\bp
Using Lemma~\ref{lem:lattice-set} it is easy to see that the functions $\un_{Z_s}$, $s\in S$, span the algebra of locally constant functions on $\widehat{E(S)}$. It follows that any probability measure on $\widehat{E(S)}$ is completely determined by its values on the sets $Z_s$. If such a measure $\mu$ satisfies~\eqref{eq:scaling}, then, as we already used in the proof of Lemma~\ref{lem:neccessary}, we have $\mu(Z_s)=N(s)^{-\beta}$. This proves the uniqueness of $\mu$.

\smallskip

In order to prove \eqref{eq:muN}, let us look at the algebra of locally constant functions on $\widehat{E(S)}$ in more detail. Every finite dimensional subspace of this algebra lies in the span of the characteristic functions of finitely many sets $Z_{s_1},\dots, Z_{s_n}$. We may assume that $e\in\{s_1,\dots,s_n\}$.  By taking intersections of the sets $s_iS$, which are either empty or have the form $sS$, and adding such elements~$s$ to the set $\{s_1,\dots,s_n\}$, we may also assume that the collection $\{s_1S,\dots,s_nS\}\cup\{\emptyset\}$ is closed under intersections. By Lemma~\ref{lem:lattice-set}, this implies that the collection $\{Z_{s_1},\dots,Z_{s_n}\}\cup\{\emptyset\}$ is also closed under intersections. Hence $A:=\operatorname{span}\{\un_{Z_{s_i}}:1\le i\le n\}$ is a finite dimensional unital C$^*$-subalgebra of $C(\widehat{E(S)})$.

On the other hand, consider the space $B:=\operatorname{span}\{\un_{I_{[s_i]}}:1\le i\le n\}$ for the sets $I_{[s]}:=\{[t]\in S/{\sim_N}: [s]\le[t]\}$.
As the functions $\un_{Z_s}$, with $s$ running over a set of representatives of $S/S^*$, are linearly independent, we have a well-defined linear map $\pi\colon A\to B$ such that $\pi(\un_{Z_{s_i}})=\un_{I_{[s_i]}}$.
By Lemma~\ref{lem:quasi-lattice}, the space $B$ is a unital C$^*$-subalgebra of $\ell^\infty(S/{\sim_N})$ and $\pi$ is a unital $*$-homomorphism. Since the functions $\un_{I_{[s]}}$, $[s]\in S/{\sim_N}$, are linearly independent, we have a well-defined linear functional $\varphi$ on $\operatorname{span}\{\un_{I_{[s]}}: [s]\in S/{\sim_N}\}$ such that $\varphi(\un_{I_{[s]}})=N(s)^{-\beta}$. Then, for any $f\in A$, we have
$$
\int f\,d\mu=\varphi(\pi(f)).
$$

In particular, if we take $s=s_i$ and $F\subset\{s_j: s_j\in s_iS\}$, then
$$
\un_{Z_{s,F}}=\prod_{r\in F}(\un_{Z_s}-\un_{Z_r}),
$$
where the product is interpreted as $\un_{Z_s}$ if $F=\emptyset$. Therefore
$$
\pi(\un_{Z_{s,F}})=\prod_{r\in [F]}(\un_{I_{[s]}}-\un_{I_{[r]}})=\un_{I_{[s]}}+\sum_{\emptyset\neq K\subset [F]} (-1)^{\vert K\vert} \un_{I_{[q_K]}},
$$
with the convention $I_\infty=\emptyset$. Applying $\varphi$ we get the required formula for $\mu(Z_{s,F})$.

\smallskip

Finally, assume that the right hand side in \eqref{eq:muN} is nonnegative for $s=e$ and all $[F]\ssubset \{[t]\in S/{\sim_N}: [e]<[t]\}$.
Since $N$ is a homomorphism, we can then also conclude that it is nonnegative for all $[s]\in S/{\sim_N}$ and $[F]\ssubset \{[t]\in S/{\sim_N}: [s]<[t]\}$.

Let $A$ and $B$ be the algebras introduced above for a given choice of $\{s_1,\ldots,s_n\}$. The atoms of~$B$ have the form $\prod_{r\in [F]}(\un_{I_{[s]}}-\un_{I_{[r]}})$, where $s=s_i$ for some~$i$ and $F=\{s_j: [s_i]<[s_j]\}$. Hence $\varphi|_B\ge0$, and therefore $\varphi\circ\pi$ is a state on~$A$. Since this is true for all $A$ as above, we conclude that we have a positive linear functional on the algebra of locally constant functions on $\widehat{E(S)}$ such that its value at $\un_{Z_s}$ is $N(s)^{-\beta}$ for all~$s$. By continuity this functional extends to a state on $C(\widehat{E(S)})$ which then defines a probability measure~$\mu$ such that $\mu(Z_s)=N(s)^{-\beta}$ for all $s$.

It remains to show that $\mu$ satisfies the scaling condition~\eqref{eq:scaling}. Thus, we have to show that, for every $s\in S$, the Borel measures $\mu(\lambda_s\cdot)$ and $N(s)^{-\beta}\mu$ are equal. Since the functions $\un_{Z_t}$, $t\in S$, span a dense subspace of $C(\widehat{E(S)})$, it suffices to check that $\mu(\lambda_s\cdot Z_t)=N(s)^{-\beta}\mu(Z_t)$ for all $t$. But this is obviously true, as $\lambda_s\cdot Z_t=Z_{st}$.
\ep

\begin{remark}
A similar characterization of $\mu_{N,\beta}$ can also be given in terms of the quasi-lattice~$S/S^*$, which does not require any assumptions on $N$ as this quasi-lattice is always well-defined. The advantage of the quasi-lattice $S/{\sim_N}$ is that its order structure is usually simpler than that of~$S/S^*$. The set $S/{\sim_N}$ is still infinite (once $N$ is nontrivial) and therefore we still need to check positivity of infinitely many expressions~\eqref{eq:muN}. It is an interesting problem to find situations when it suffices to check only finitely many inequalities. For example, it is shown in~\cite{ALN} that the finitely generated right-angled Artin monoids have this property.
\end{remark}

\begin{remark}\label{rem:beta1enough}
We have $\mu_{N,\beta}=\mu_{N^\beta}$, so in developing a general theory for a fixed $\beta$ we can always assume $\beta=1$.
\end{remark}

\section{KMS-states of finite type}\label{sec:KMS}

Let $S$ be a right LCM semigroup and $N\colon S\to[1,+\infty)$ be a nontrivial scale. Assuming that the partially ordered set $(\ker N)/(\ker N)^*$ is directed, so that the equivalence relation~\eqref{eq:equiv} on $S$ for $T=\ker N$ is well-defined, the corresponding \emph{partition function}, or \emph{$\zeta$-function}, is defined by
$$
\zeta_N(\beta) := \sum_{[s] \in S/{\sim_N}} N(s)^{-\beta}.
$$
A KMS$_\beta$-state will be called of \emph{finite type} if $\zeta_N(\beta)<+\infty$. Our goal is to obtain a complete description of such KMS-states.

\smallskip

A necessary condition for the existence of KMS$_\beta$-states for $\beta\ne0$ is that $N$ satisfies condition~\eqref{eq:admis}, see Lemma~\ref{lem:neccessary}. In the present case this condition is automatically satisfied.

\begin{lemma}\label{lem:automatic_admis}
If $\zeta_N(\beta)<+\infty$ for some $\beta$, then $N$ satisfies condition~\eqref{eq:admis}.
\end{lemma}

\bp
For $\zeta_N(\beta)$ to be finite, at the very least we need the sets $N^{-1}(\lambda)/{\sim_N}$, $\lambda\in N(S)$, to be finite. But this implies condition~\eqref{eq:admis} by Lemma~\ref{lem:finite-level}.
\ep

Before we move to the KMS-states, let us prove a couple of general results on hereditary LCM submonoids.

Let us introduce the following notation. Assume $T$ is a hereditary right LCM submonoid of $S$. Then $E(T)$ is a subsemigroup of $E(S)$. We can extend every character $\chi$ of $E(T)$ to $E(S)$ by letting
$$
\chi'(e_s):=0\ \ \text{for}\ \ s\in S\setminus T.
$$
It is easy to see that $\chi'$ is a character of $E(S)$.

Next, recall that for any inverse semigroup $I$ without $0$ the relation
$$
x\gamma y\ \ \text{iff}\ \ xp=yp\  \ \text{for some idempotent}\ \ p \in I
$$
is the congruence that defines the \emph{maximal group homomorphic image} $I/\gamma$ of $I$.

\begin{thm}\label{thm:bisimple inverse submonoids}
Let $T$ be a hereditary right LCM submonoid of a right LCM semigroup $S$. Then the map $[\lambda_s^{\phantom{1}}\lambda_t^{-1},\chi]\mapsto [\lambda_s^{\phantom{1}}\lambda_t^{-1}, \chi']$ gives an embedding $\iota$ of $\CG(T)$ into $\CG(S)$ whose image coincides with the reduction of $\CG(S)$ by $\iota(\CG(T)^{(0)})$.

Suppose in addition that the partially ordered set $T/T^*$ is directed. Then the following statements hold:
\begin{enumerate}
\item[(i)] for the maximal character $\un\in \widehat{E(T)}$ (given by $\un(e_t) = 1$ for all $t$) and $\chi_T:=\un'\in\widehat{E(S)}$, we have $\CG(S)^{\chi_T}_{\chi_T} = \{[\lambda_s^{\phantom{*}}\lambda_t^{-1},\chi_T] \mid s,t \in T\}$, 
and this isotropy group is canonically isomorphic to $I_\ell(T)/\gamma$, where $I_\ell(T)$ now denotes the left inverse hull of $T$ without $0$;
\item[(ii)] the $\CG(S)$-orbit of $\chi_T$ is  given by $\CO(\chi_T) = \{ \chi_{T,s} \mid s \in S\}$, where $\chi_{T,s}(e_r) = 1$ if and only if $sr' \in rS$ for some $r' \in T$;
\item[(iii)] the map $S/{\sim_T} \to \CO(\chi_T), [s] \mapsto \chi_{T,s}$, where $\sim_T$ is the equivalence relation on $S$ defined by~\eqref{eq:equiv}, is a bijection.
\end{enumerate}
\end{thm}

\begin{proof}
It is easy to see that the map $\iota\colon\CG(T)\to\CG(S)$ is well-defined. In order to see that it is a homomorphism of groupoids it suffices to check that
$$
\lambda_s^{\phantom{1}}\lambda_t^{-1}\cdot\chi'=(\lambda_s^{\phantom{1}}\lambda_t^{-1}\cdot\chi)'
$$
for all $s,t\in T$ and $\chi\in\widehat{E(T)}$ with $\chi(e_t)=1$, where on the left hand side we of course consider $\lambda_s^{\phantom{1}}\lambda_t^{-1}$ as an element of~$I_\ell(S)$, while on the right hand side as an element of $I_\ell(T)$. For this it suffices to check that $\lambda_s^{\phantom{1}}\lambda_t^{-1}\cdot\chi'$ vanishes on $e_r$ for $r\in S\setminus T$. For this, in turn, it is enough to show that if $\lambda_t^{\phantom{1}}\lambda_s^{-1}e_r\lambda_s^{\phantom{1}}\lambda_t^{-1}=e_a$ for some $s,t,a\in T$, then $r\in T$. The equality $\lambda_t^{\phantom{1}}\lambda_s^{-1}e_r\lambda_s^{\phantom{1}}\lambda_t^{-1}=e_a$ means that $a=tb$ for some $b\in S$ such that $rS\cap sS=sbS$. Using that $T$ is a hereditary right LCM submonoid we first conclude that $b\in T$ (see Remark~\ref{rem:hereditary}) and then that $r\in T$.

Next, let us check that $\iota$ is injective. Assume $[\lambda_{s_1}^{\phantom{1}}\lambda_{t_1}^{-1},\chi']=[\lambda_{s_2}^{\phantom{1}}\lambda_{t_2}^{-1},\chi']$ in $\CG(S)$ for some $s_i,t_i\in T$ and $\chi\in\widehat{E(T)}$. By Lemma~\ref{lem:equivalence relation explained} this means that there exist $r_1,r_2\in S$ such that $s_1r_1=s_2r_2$, $t_1r_1=t_2r_2$ and $\chi'(e_{t_1r_1})=1$. The last equality means that $t_1r_1\in T$. Hence $r_1\in T$ and, similarly, $r_2\in T$. Therefore we get $[\lambda_{s_1}^{\phantom{1}}\lambda_{t_1}^{-1},\chi]=[\lambda_{s_2}^{\phantom{1}}\lambda_{t_2}^{-1},\chi]$ in $\CG(T)$.

Thus, we can consider $\CG(T)$ as a subgroupoid of $\CG(S)$, at least if we ignore the topologies on both groupoids. It is clear that $\CG(T)$ is contained in the reduction of $\CG(S)$ by $\iota(\CG(T)^{(0)})$. Assume an element $[\lambda_s^{\phantom{1}}\lambda_t^{-1},\chi']$ lies in this reduction. Then $\chi'(e_t)=1$, hence $t\in T$. Since the inverse of $[\lambda_s^{\phantom{1}}\lambda_t^{-1},\chi']$ lies in the reduction as well, we similarly get $s\in T$. Thus, $\CG(T)$ is indeed the reduction of $\CG(S)$ by $\iota(\CG(T)^{(0)})$. By the definition of the topology on $\CG(S)$ it then also becomes clear that the topology on $\CG(T)$ coincides with the relative topology.

\smallskip

Assume now that $T/T^*$ is directed. To prove (i), we may assume that $T=S$ because the isotropy group of $\chi_T$ is contained in the reduction of $\CG(S)$ by $\iota(\CG(T)^{(0)})$. But then the first part of~(i) becomes obvious, since $T/T^*$ is directed and as a result $x\cdot\un=\un$ for all $x\in I_\ell(T)$. It remains to check that $I_\ell(T)/\gamma=\CG(T)^\un_\un$.

First of all, as $\gamma$ is by construction the congruence that is satisfied by every homomorphism from an inverse semigroup to a group, we have a homomorphism $I_\ell(T)/\gamma\to\CG(T)^\un_\un$. This map is clearly surjective. It is also injective, because the condition $[x,\un]=[y,\un]$ in $\CG(T)$ for $x,y \in I_\ell(T)$ is equivalent to the existence of $t \in T$ with  $\un(e_t)=1$ and $xe_t=ye_t$, which is exactly the condition~$x\gamma y$.

\smallskip

For (ii),  the $\CG(S)$-orbit of $\chi_T$ consists of the points $\lambda_s^{\phantom{1}}\lambda_t^{-1}\cdot\chi_T$ for all $s,t\in S$ such that the element $[\lambda_s^{\phantom{1}}\lambda_t^{-1},\chi_T]$ of $\CG(S)$ is well-defined. This means that $\chi_T(e_t)=1$, that is, $t\in T$. Since the elements $\lambda_t$, $t\in T$, stabilize $\chi_T$, the orbit therefore consists of the elements $\lambda_s\cdot\chi_T$, $s\in S$.

Take $r\in S$, then the element $\lambda_s^{-1}e_r\lambda_s$ equals $e_{r'}$ if $rS\cap sS=sr'S$, and it equals $0$ if $rS\cap sS=\emptyset$. Hence $(\lambda_s\cdot\chi_T)(e_r)=1$ if and only if $rS\cap sS=sr'S$ for some $r'\in T$. The last condition is equivalent to the existence of $r'\in T$ such that $sr'\in rS$. Thus, $\lambda_s\cdot\chi_T=\chi_{T,s}$.

\smallskip

Turning to (iii), recall again that the elements $\lambda_t$, $t\in T$, stabilize $\chi_T$. It follows that the character $\chi_{T,s}=\lambda_s\cdot\chi_T$ depends only on the equivalence class of $s\in S$. Therefore the map $S/{\sim_T}\to\CO(\chi_T)$, $[s]\mapsto\chi_{T,s}$, is well-defined and surjective.

Assume now that $\chi_{T,s}=\chi_{T,t}$ for some $s,t \in S$. As $\chi_{T,s}(s)=\chi_{T,t}(t)=1$, there exist $c,d \in T$ with $sc \in tS$ and $td \in sS$. It follows that $sS\cap tS = saS$ and $sa=tb$ for some $a,b \in S$ with $c \in aS$ and $d\in bS$. As $T$ is hereditary, this forces $a,b \in T$. Hence $s\sim_T t$.
\end{proof}

\begin{remark}\label{rem:rLCM 0-bisimple}
In view of the categorical equivalences between bisimple inverse monoids and directed right LCM semigroups, see \cite{Cli}, and $0$-bisimple inverse monoids and right LCM semigroups, see \cite{Law0}, Theorem~\ref{thm:bisimple inverse submonoids} can be understood as a result on hereditary bisimple inverse submonoids of $0$-bisimple inverse monoids. Here $J\subset I$ is hereditary if, for every $x \in I$, the existence of an idempotent $p \in J$ with $xx^{-1}p= p$ forces $x \in J$.
\end{remark}

\begin{corollary}\label{cor:honest subalgebra}
Let $T$ be a hereditary right LCM submonoid of a right LCM semigroup $S$. Then the natural $*$-homomorphism $C^*(T) \to C^*(S)$ arising from the inclusion $T\subset S$ is faithful, so we can consider $C^*(T)$ as a C$^*$-subalgebra of $C^*(S)$. Furthermore, there is a conditional expectation $E\colon C^*(S)\to C^*(T)$ such that
$$
E(v_s^{\phantom{*}}v_t^*)=\begin{cases}v_s^{\phantom{*}}v_t^*,&\text{if}\ \ s,t\in T,\\ 0,&\text{otherwise}.\end{cases}
$$
\end{corollary}

\begin{proof}
Denote by $\pi$ the canonical $*$-homomorphism $C^*(T) \to C^*(S)$.

By Theorem~\ref{thm:bisimple inverse submonoids}, $\iota(\CG(T))$ is the reduction of $\CG(S)$ by $\iota(\CG(T)^{(0)})$. In particular, this reduction is an \'etale groupoid.
Hence, by \cite{LLN}*{Remark~1.6}, the restriction map $C_c(\CG(S))\to C_c(\CG(S)_{\iota(\CG(T)^{(0)})})$ extends to a completely positive contraction $E\colon C^*(\CG(S))\to C^*(\CG(S)_{\iota(\CG(T)^{(0)})})$. It should be stressed that this result is formulated in~\cite{LLN} for Hausdorff groupoids, but its proof can be extended to the non-Hausdorff case as well.

We identify $C^*(S)$ with $C^*(\CG(S))$ and $C^*(T)$ with $C^*(\CG(T))$, and therefore consider $E$ as a map $C^*(S)\to C^*(T)$.
By construction it is clear that $E\circ\pi=\operatorname{id}$. Hence $\pi$ is injective, so we can consider~$C^*(T)$ as a subalgebra of $C^*(S)$. But then $E$, being a completely positive contractive projection, becomes a conditional expectation.

Finally, as was shown in the proof of Theorem~\ref{thm:bisimple inverse submonoids}, we have $[\lambda_s^{\phantom{1}}\lambda_t^{-1},\chi'] \in \CG(S)_{\iota(\CG(T)^{(0)})}$ only if $s,t\in T$, hence $E(v_s^{\phantom{*}}v_t^*)\ne0$ only if $s,t\in T$.
\end{proof}

We now apply the above results to $T=\ker N$. To simplify the notation, we will denote the character $\chi_{\ker N}$ on $E(S)$ by $\chi_N$, so that
$$
\chi_N(e_s)=1\ \ \text{if and only if}\ \ s\in\ker N.
$$

\begin{thm}\label{thm:finiteKMS}
Let $S$ be a countable right LCM semigroup and $N$ be a nontrivial scale on $S$ such that $(\ker N)/(\ker N)^*$ is directed. Assume the abscissa of convergence $\beta_c$ of the series $\sum_{[s] \in S/{\sim_N}} N(s)^{-\beta}$ is finite. Then, for every $\beta>\beta_c$, there exists an affine homeomorphism between the $\sigma^N$-KMS$_\beta$-states on~$C^*(S)$ and the tracial states on $C^*(\ker N)$, or equivalently, on the group C$^*$-algebra $C^*(I_\ell(\ker N)/\gamma)$.

Explicitly, the KMS$_\beta$-state $\varphi_{\tau,\beta}$ corresponding to a tracial state $\tau$ on $C^*(\ker N)$ is given by
\begin{equation} \label{eq:KMS3}
\varphi_{\tau,\beta}(v_s^{\phantom{*}}v_t^*)=\frac{N(s)^{-\beta}}{\zeta_N(\beta)}\sum_{[x]\in S/{\sim_N}:[sx]=[tx]}N(x)^{-\beta}\tau(v_{q_x}^{\phantom{*}}v_{p_x}^*),
\end{equation}
where $p_x,q_x\in\ker N$ are such that $sxp_x=txq_x$.
\end{thm}

Here $I_\ell(\ker N)$ denotes again the inverse hull of $\ker N$ without $0$. Note also that since we assume that $N$ is nontrivial, we must have $\beta_c\ge0$.

\bp
Fix $\beta>\beta_c$. By Lemma~\ref{lem:automatic_admis}, the scale $N$ satisfies condition~\eqref{eq:admis}. Hence we can apply Theorem~\ref{thm:bisimple inverse submonoids} to $T:=\ker N$, which in particular yields a bijection $S/{\sim_N} \to \CO(\chi_N)$, $[s] \mapsto \lambda_s\cdot\chi_N=\chi_{N,s}$.

Define a probability measure $\mu$ on $\widehat{E(S)}$ by $\mu=\zeta_N(\beta)^{-1}\sum_{[s]\in S/{\sim_N}}N(s)^{-\beta}\delta_{\chi_{N,s}}$. Since $\lambda_r\cdot\chi_{N,s}=\lambda_{rs}\cdot\chi_N=\chi_{N,rs}$, it is clear that $\mu$ satisfies the scaling condition~\eqref{eq:scaling}. By Proposition~\ref{prop:muN} there is at most one such measure, so $\mu=\mu_{N,\beta}$. Since $\mu_{N,\beta}$ is concentrated on the $\CG(S)$-orbit of $\chi_N$, a measurable field of tracial states $(\tau_x)_x$ as in Theorem~\ref{thm:KMS} is completely determined by one trace~$\tau_x$ for any point $x$ on the orbit, cf.~\cite{Nes}*{Corollary~1.4}. Thus we conclude that there exists an affine homeomorphism between the $\sigma^N$-KMS$_\beta$-states on $C^*(S)=C^*(\CG(S))$ and the tracial states on $C^*(\CG(S)^{\chi_N}_{\chi_N})$. But by Theorem~\ref{thm:bisimple inverse submonoids} we have $C^*(\CG(S)^{\chi_N}_{\chi_N})\cong C^*(I_\ell(\ker N)/\gamma)$.

The canonical map $C^*(\ker N)\to C^*(I_\ell(\ker N)/\gamma)$ induces a homeomorphism of the corresponding trace state spaces. This can be deduced again from Theorem~\ref{thm:KMS} applied to $\CG(\ker N)$ and the trivial dynamics, but is also easy to see directly. Indeed, if $\tau$ is a tracial state on $C^*(\ker N)$, then in the corresponding GNS-representation the partial isometries $v_s$, $s\in\ker N$, become unitaries. Hence the GNS-representation of $C^*(\ker N)$ defines a representation of $I_\ell(\ker N)/\gamma$ and therefore~$\tau$ factors through $C^*(I_\ell(\ker N)/\gamma)$.

It remains to prove formula~\eqref{eq:KMS3}. Recalling the explicit form~\eqref{eq:groupoid-iso} of the isomorphism $C^*(S)\cong C^*(\CG(S))$,
we have
\begin{equation} \label{eq:KMS2}
\varphi_{\tau,\beta}(v_s^{\phantom{*}}v_t^*)=\frac{1}{\zeta_N(\beta)}\sum_{[y]}N(y)^{-\beta}\tau(\pi_y(u_{[\lambda_s^{\phantom{1}}\lambda_t^{-1},\chi_{N,y}]})),
\end{equation}
where the summation is over $[y]\in S/{\sim_N}$ such that $\lambda_s^{\phantom{1}}\lambda_t^{-1}\cdot\chi_{N,y}=\chi_{N,y}$, $\pi_y$ denotes the isomorphism $C^*(\CG(S)^{\chi_{N,y}}_{\chi_{N,y}})\to C^*(\CG(S)^{\chi_{N}}_{\chi_{N}})$ obtained by conjugation by the element $[\lambda_y^{-1},\chi_{N,y}]\in\CG(S)^{\chi_{N}}_{\chi_{N,y}}$ and we view the trace $\tau$ on $C^*(\ker N)$ as a trace on $C^*(\CG(S)^{\chi_{N}}_{\chi_{N}})$.

Assume $[y]\in S/{\sim_N}$ contributes to~\eqref{eq:KMS2}. In other words, $y\sim_N sx\sim_N tx$ for some $x\in S$. We may take $y=tx$. Let $p_x,q_x \in \ker N$ satisfy $sxp_x=txq_x$. We claim that then
$$
[\lambda_y^{-1},\chi_{N,y}]\,[\lambda_s^{\phantom{1}}\lambda_t^{-1},\chi_{N,y}]\,[\lambda_y^{-1},\chi_{N,y}]^{-1}=[\lambda_{q_x}^{\phantom{1}}\lambda_{p_x}^{-1},\chi_N].
$$
Indeed, we have
$$
\lambda_y^{-1}\lambda_s^{\phantom{1}}\lambda_t^{-1}\lambda_ye_{p_x}=\lambda_y^{-1}\lambda_s^{\phantom{1}}\lambda_x^{\phantom{1}}e_{p_x}
=\lambda_y^{-1}\lambda_{sxp_x}^{\phantom{1}}\lambda_{p_x}^{-1}e_{p_x}
=\lambda_y^{-1}\lambda_{tx}^{\phantom{1}}\lambda_{q_x}^{\phantom{1}}\lambda_{p_x}^{-1}e_{p_x}=\lambda_{q_x}^{\phantom{1}}\lambda_{p_x}^{-1}e_{p_x}.
$$
Since $\lambda_y\cdot\chi_N=\chi_{N,y}$ and $\chi_N(e_{p_x})=1$, this proves our claim.

Recalling again the isomorphism~\eqref{eq:groupoid-iso}, but now for the semigroup $\ker N$, we conclude that
$$
N(y)^{-\beta}\tau(\pi_y(u_{[\lambda_s^{\phantom{1}}\lambda_t^{-1},\chi_{N,y}]}))=N(s)^{-\beta}N(x)^{-\beta}\tau(v_{q_x}^{\phantom{*}}v_{p_x}^*).
$$
Plugging this into~\eqref{eq:KMS2} we get formula~\eqref{eq:KMS3} for $\varphi_{\tau,\beta}$.
\ep

\begin{example}
Consider the semigroup $S=\Z_+\rtimes\N$ and the scale $N$ from Example~\ref{ex:ax+b}. The corresponding $\zeta$-function is
$$
\zeta_N(\beta)=\sum^\infty_{n=1}n^{1-\beta}=\zeta(\beta-1),
$$
where $\zeta$ is the Riemann $\zeta$-function. Therefore $\beta_c=2$. The C$^*$-algebra $C^*(\ker N)=C^*(\Z_+)$ is the Toeplitz algebra $\mathcal T$. Its tracial states are given by the probability measures on the unit circle. Namely, the trace corresponding to such a measure $\nu$ is defined by $\tau(v_k^{\phantom{*}}v_l^*)=\int_\T z^{k-l}d\nu(z)$.

Let us compute the $\sigma^N$-KMS$_\beta$-state $\varphi_{\nu,\beta}$ corresponding to $\nu$ for a fixed $\beta>2$. As $\varphi_{\nu,\beta}(v_s^{\phantom{1}}v_s^*)=N(s)^{-\beta}$ and $\varphi_{\nu,\beta}(v_s^{\phantom{1}}v_t^*)=0$ if $N(s)\ne N(t)$, the only interesting case is when $s=(c,n)$ and $t=(d,n)$, with $c\ne d$. By taking the adjoint of $v_s^{\phantom{*}}v_t^*$ if necessary, we may assume $c>d$.

Recall from Example~\ref{ex:ax+b} that $S/{\sim_N}=\bigsqcup^\infty_{m=1}\Z/m\Z$, with $[x]\in \Z/m\Z$ being the class of the element $(x,m)\in S$ (if $x\ge0$). For $[x]\in \Z/m\Z$, we have $[c+nx]=[d+nx]$ in $\Z/nm\Z$ if and only if $nm|(c-d)$, and then
$$
(c,n)(x,m)=(d,n)(x,m)\Big(\frac{c-d}{nm},1\Big)\ \ \text{in}\ \ S.
$$
By~\eqref{eq:KMS3} it follows that, for $n\ge1$ and $c>d\ge0$, we have
$$
\varphi_{\nu,\beta}(v_{(c,n)}^{\phantom{*}}v_{(d,n)}^*)=\frac{n^{-\beta}}{\zeta(\beta-1)}\sum_{m:nm|(c-d)}m^{1-\beta}\int_\T z^{\frac{c-d}{nm}}d\nu(z).
$$
This is exactly the formula from~\cite{LR}*{Theorem~7.1(3)}.

The restriction of $\varphi_{\nu,\beta}$ to $C^*(\ker N)$ is a tracial state, hence it is also given by a probability measure $\tilde\nu$ on $\T$. This measure is characterized by the property
$$
\int_\T z^kd\tilde\nu(z)=\frac{1}{\zeta(\beta-1)}\sum_{m|k}m^{1-\beta}\int_\T z^{\frac{k}{m}}d\nu(z)\ \ \text{for all}\ \ k\ge0.
$$
It is not difficult to check that it can also be described as follows.

For $n\in\N$, consider the transfer operator $\LL_n\colon C(\T)\to C(\T)$,
$$
(\LL_nf)(z)=\sum_{w:w^n=z}f(w).
$$
We then get the dual transfer operator $\LL_n^*$ of norm $n$ acting on the space $M(\T)=C(\T)^*$ of complex Borel measures on~$\T$. Then
$$
\tilde\nu=\frac{1}{\zeta(\beta-1)}\sum^\infty_{n=1}n^{-\beta}\LL_n^*\nu=\frac{1}{\zeta(\beta-1)}\prod_{p\in\mathcal P}(1-p^{-\beta}\LL_p^*)^{-1}\nu,
$$
where $\mathcal P$ is the set of prime numbers. The state  $\varphi_{\nu,\beta}$ is completely determined by the measure $\tilde\nu$ as well (this is true in greater generality, see Lemma~\ref{lem:ker-reduction}), but whereas $\nu$ can be arbitrary, $\tilde\nu$ can not: the restriction we get is that the signed measure $\prod_{p\in\mathcal P}(1-p^{-\beta}\LL_p^*)\tilde\nu$ must be positive.
\hfill$\diamondsuit$
\end{example}

\begin{example}\label{ex:AT}
Given a finite set $\Pi$ and a symmetric matrix $M=(m_{s,t})_{s,t\in\Pi}$ such that $m_{s,t}\in\N\cup\{\infty\}$, $m_{s,s}=1$ and $m_{s,t}\ge2$ for $s\ne t$, consider the corresponding Artin--Tits monoid $A^+_M$ with generators $s\in\Pi$ and relations
$$
\underbrace{sts\dots}_{m_{s,t}}=\underbrace{tst\dots}_{m_{s,t}}\ \ \text{if}\ \ m_{s,t}<\infty.
$$
It is known that $A^+_M$ is a right LCM semigroup, see~\cite{BS}*{Proposition~4.1}.

Assume next that $G$ is a discrete group acting by permutations $s\mapsto g.s$ on $\Pi$ such that $m_{g.s,g.t}=m_{s,t}$ for all $s,t\in\Pi$. Then $G$ acts on $A^+_M$ by permuting the generators $s\in\Pi$. Consider the semidirect product $S:=A^+_M\rtimes G$. It is easy to see that $S$ is a right LCM semigroup.

Consider now a map $N\colon\Pi\to(1,+\infty)$ satisfying $N(g.s)=N(s)$ for all $s\in\Pi$ and $g\in G$, and $N(s)=N(t)$ for all $s,t\in\Pi$ such that $m_{s,t}$ is an odd number. It extends to a scale on $S$ that is trivial on $G$. Then $\ker N=G$ and $S/{\sim_N}$ can be identified with $A^+_M$. Since $\Pi$ is finite and $A^+_M$ is a quotient of the free monoid on the set $\Pi$, it follows that $\beta_c<+\infty$. Specifically, $\beta_c\le\beta_0$, where $\beta_0$ is such that
$$
\sum_{s\in\Pi}N(s)^{-\beta_0}=1.
$$
By Theorem~\ref{thm:finiteKMS} we conclude that for every $\beta>\beta_c$ the $\sigma^N$-KMS$_\beta$-states on $C^*(S)$ are in an affine bijective correspondence with the tracial states on $C^*(G)$. In particular, if $G$ is either finite or abelian, then the extremal KMS$_\beta$-states are parameterized by the equivalence classes of irreducible representations of~$G$.

\smallskip

This example can be modified by replacing $G$ by a right LCM semigroup $P$ with quotient $G$. The new monoid $A^+_M\rtimes P$ might not always be LCM, but when it is, we get that for $\beta>\beta_c$ the KMS$_\beta$-states on $C^*(S)$ are in an affine bijective correspondence with the tracial states on $C^*(P)$.

For example, consider the simplest case $A^+_M=\Z^2_+$, $P=\Z_+$ and $1\in P$ acts on $\Z^2_+$ by flipping the coordinates. The corresponding monoid $\Z^2_+\rtimes\Z_+$ is the right LCM semigroup $C_3$ considered in~\cite{Gar}. It has generators $x_1$, $x_2$ and $x_3$ and relations
$$
x_1x_2=x_2x_1,\ \ x_3x_1=x_2x_3,\ \ x_3x_2=x_1x_3.
$$
We then consider a scale $N$ such that $N(x_1)=N(x_2)>1$ and $N(x_3)=1$. Then $\beta_c=0$, and the conclusion is that for every $\beta>0$ the $\sigma^N$-KMS$_\beta$-states on $C^*(C_3)$ are in an affine bijective correspondence with the tracial states on $C^*(\Z_+)$, that is, with the probability measures on $\T$.
\hfill$\diamondsuit$
\end{example}

Recall that, given a one-parameter group of automorphisms $\sigma$ of a C$^*$-algebra $A$, a $\sigma$-KMS$_\infty$-state is the weak$^*$ limit of $\sigma$-KMS$_{\beta_i}$-states $\varphi_i$ with $\beta_i\to+\infty$.

\smallskip

Letting $\beta\to+\infty$ in~\eqref{eq:KMS3} we immediately get the following.

\begin{corollary}\label{cor:KMSinfty}
Under the assumptions of Theorem~\ref{thm:finiteKMS}, there exists an affine homeomorphism between the $\sigma^N$-KMS$_\infty$-states on $C^*(S)$ and the tracial states on $C^*(\ker N)$. Namely, the KMS$_\infty$-state $\varphi_{\tau,\infty}$ corresponding to a trace $\tau$ on $C^*(\ker N)$ is given by
$$
\varphi_{\tau,\infty}(v_s^{\phantom{*}}v_t^*)=\begin{cases}\tau(v_s^{\phantom{*}}v_t^*),&\text{if}\ \ s,t\in \ker N,\\ 0,&\text{otherwise}.\end{cases}
$$
\end{corollary}

Let us now for completeness determine the structure of ground states on $(C^*(S),\sigma_N)$, for which we will apply \cite{LLN}*{Theorem~1.4} (although this is not really needed once we accept Corollary~\ref{cor:honest subalgebra} above for $T=\ker N$).

Recall from \cite{LLN} that the \emph{boundary set} of the cocycle $c_N$ defining the dynamics $\sigma^N$ on  $C^*(S)=C^*(\CG(S))$ is given by
$$
Z:= \{ x \in \CG(S)^{(0)} \mid c_N \geq 0 \text{ on } \CG(S)_x \}.
$$
The corresponding reduction groupoid $\CG(S)_Z$ is referred to as the \emph{boundary groupoid}.

By $c^{-1}_N(0)$ we denote the kernel groupoid for $c_N$.

\begin{lemma}\label{lem:boundary and kernel groupoids}
Let $S$ be a right LCM semigroup and $N$ be a scale on $S$ satisfying~\eqref{eq:admis}. Then
\begin{enumerate}
\item[(i)] $c_N^{-1}(0)= \{ [\lambda_s^{\phantom{*}}\lambda_t^{-1},\chi] \in \CG(S) \mid N(s)=N(t)\}$, and this groupoid is \'{e}tale;
\item[(ii)] $\CG(S)_Z = \iota(\CG(\ker N))$, where $\iota\colon \CG(\ker N)\to\CG(S)$ is the embedding given by Theorem~\ref{thm:bisimple inverse submonoids}.
\end{enumerate}
\end{lemma}

\begin{proof}
The description of $c_N^{-1}(0)$ in (i) is clear. This groupoid is \'{e}tale because $\CG(S)$ is and $c_N^{-1}(0)$ is an open subset of $\CG(S)$.

To prove (ii), by Theorem~\ref{thm:bisimple inverse submonoids} we just have to check that $Z=\iota(\CG(\ker N)^{(0)})$. Suppose $\psi \in \widehat{E(S)}$ satisfies $c_N([\lambda_s^{\phantom{*}}\lambda_t^{-1},\psi]) = \log N(s)-\log N(t)\geq 0$ for all $s,t\in S$ with $\psi(e_t)=1$. This forces $t \in \ker N$, so that $\psi$ vanishes on the elements $e_s$ for $s\in S\setminus\ker N$, that is, $\psi=\chi'$ for some $\chi \in \widehat{E(\ker N)}$. Conversely, it is clear that for every $\chi\in\widehat{E(\ker N)}$ the character~$\chi'$ belongs to~$Z$.
\end{proof}

In view of this lemma we can apply \cite{LLN}*{Theorem~1.4} (extended to the non-Hausdorff case) and get the following result.

\begin{proposition}\label{prop:ground states}
Let $S$ be a countable right LCM semigroup and $N$ be a scale on $S$ satisfying~\eqref{eq:admis}. Then there exists an affine homeomorphism between the states on $C^*(\ker N)$ and the ground states on $C^*(S)$ with respect to $\sigma^N$. Namely, the ground state corresponding to a state $\psi$ on $C^*(\ker N)$ is $\psi\circ E$, where $E\colon C^*(S)\to C^*(\ker N)$ is the conditional expectation defined in Corollary~\ref{cor:honest subalgebra}.
\end{proposition}

\begin{corollary}
Under the assumptions of Theorem~\ref{thm:finiteKMS}, every ground state on $C^*(S)$ with respect to $\sigma^N$ is a KMS$_\infty$-state if and only if $\ker N = S^*$ and $S^*$ is abelian.
\end{corollary}

\bp
The descriptions of KMS$_\infty$-states in Corollary~\ref{cor:KMSinfty} and ground states in Proposition~\ref{prop:ground states} imply that the ground state corresponding to a state $\psi$ on $C^*(\ker N)$ is a KMS$_\infty$-state if and only if $\psi$ is tracial.  Therefore the corollary follows from the simple claim that every state on $C^*(\ker N)$ is tracial, or equivalently, $C^*(\ker N)$ is abelian, if and only if $\ker N$ is an abelian group.
\ep

\section{Uniqueness of KMS-states} \label{sec:unique}

We continue to consider a right LCM semigroup $S$ and a scale $N\colon S\to[1,+\infty)$. In this section our goal is to understand when for a fixed $\beta>0$ there exists exactly one $\sigma^N$-KMS$_\beta$-state on $C^*(S)$. Due to Remark~\ref{rem:beta1enough}, we may assume $\beta=1$ without loss of generality.

\smallskip

Assume that $S$ is countable and that $\widehat{E(S)}$ admits a probability measure $\mu_N=\mu_{N,1}$ with $\mu_N(Z_s)=N(s)^{-1}$ (according to Proposition~\ref{prop:muN}, such a measure is then unique). Equivalently, we can assume that there is at least one $\sigma^N$-KMS$_1$-state.
In view of Lemma~\ref{lem:neccessary}, it follows that the scale $N$ satisfies condition~\eqref{eq:admis}.

For $a,b\in S$, consider the sets
$$
\Omega^{a,b}_\triv:=\{\chi\in \widehat{E(S)}\mid [\lambda_a^{\phantom{1}}\lambda_b^{-1},\chi]=\chi\}\ \ \text{and}\ \
\Omega^{a,b}_\fix:=\{\chi\in \widehat{E(S)}\mid \lambda_a^{\phantom{1}}\lambda_b^{-1}\cdot\chi=\chi\}.
$$
In other words, $\Omega^{a,b}_\fix$ consists of the points fixed by $\lambda_a^{\phantom{1}}\lambda_b^{-1}$ and $\Omega^{a,b}_\triv\subset\Omega^{a,b}_\fix$ is the subset of trivially fixed points, in the sense that $\lambda_a^{\phantom{1}}\lambda_b^{-1}$ defines the unit element of the isotropy group.
Then the two KMS$_1$-states $\varphi'$ and $\varphi''$ from~\eqref{eq:2states 1} and \eqref{eq:2states 2} can be described in terms of the generators of~$C^*(S)$~by
\begin{equation}\label{eq:phis}
\varphi'(v_a^{\phantom{*}}v_b^*)=\mu_N(\Omega^{a,b}_\triv)\ \ \text{and}\ \ \varphi''(v_a^{\phantom{*}}v_b^*)=\mu_N(\Omega^{a,b}_\fix).
\end{equation}

By Corollary~\ref{cor:uniqueKMS}, a $\sigma^N$-KMS$_1$-state on $C^*(S)$ is unique if and only if
$$
\mu_N(\Omega^{a,b}_\fix\setminus\Omega^{a,b}_\triv)=0\ \ \text{for all}\ \ a,b\in S,\ \ a\ne b,
$$
that is, $\varphi'=\varphi''$.

\smallskip

Let us try to compute the measures of the sets $\Omega^{a,b}_\triv$ and $\Omega^{a,b}_\fix$.

We start with the easier case of~$\Omega^{a,b}_\triv$.
For $a,b\in S$, define
$$
B^{a,b}:=\{[s]\in S/{\sim_N}\mid s'=at=bt\ \text{for some}\ s'\sim_Ns\ \text{and}\ t\in S\}.
$$

\begin{lemma} \label{lem:Aab}
For any $a,b\in S$, we have
$$
\Omega^{a,b}_\triv=\bigcup_{[s]\in B^{a,b}}Z_s
$$
modulo a set of measure zero. Hence
\begin{equation}\label{eq:Aab}
\mu_N(\Omega^{a,b}_\triv)=\lim_{[F]}\sum_{\emptyset\neq K\subset [F]} (-1)^{\vert K\vert+1} N(q_K)^{-1},
\end{equation}
where the limit is taken over the finite subsets $[F]$ of $B^{a,b}$ ordered by inclusion.
\end{lemma}

Note that the formulation of the lemma is meaningful, since if $s\sim_N s'$ then $Z_s=Z_{s'}$ modulo a set of measure zero.

\bp
By Lemma~\ref{lem:equivalence relation explained} we have $[\lambda_a^{\phantom{1}}\lambda_b^{-1},\chi]=\chi$ if and only if there is $t\in S$ such that $at=bt$ and $\chi\in Z_{bt}$. Therefore
$$
\Omega^{a,b}_\triv=\bigcup_{s\in \tilde B^{a,b}} Z_s,
$$
where $\tilde B^{a,b}$ consists of the elements $s$ such that $s=at=bt$ for some $t$. As the image of $\tilde B^{a,b}$ in $S/{\sim_N}$ is $B^{a,b}$, this proves the first statement of the lemma.

The complement of $\bigcup_{[s]\in B^{a,b}}Z_s$ is the intersection $\bigcap_{[F]\ssubset B^{a,b}}Z_{e,F}$.
Hence
$$
\mu_N(\Omega^{a,b}_\triv)=1-\lim_{[F]\ssubset B^{a,b}}\mu_N(Z_{e,F}),
$$
from which we get the required expression for $\mu_N(\Omega^{a,b}_\triv)$ by applying~\eqref{eq:muN}.
\ep

\begin{remark}
From the proof we see that the set $\Omega^{a,b}_\triv$ is open. In general it is a proper subset of the interior of $\Omega^{a,b}_\fix$. For example, if $S$ is a group, then $Z_e=\widehat{E(S)}$ consists of one point, $\Omega^{a,b}_\fix=Z_e$ for all $a,b$, but $\Omega^{a,b}_\triv=\emptyset$ for $a\ne b$.

\end{remark}

The formula for $\mu_N(\Omega^{a,b}_\fix)$ is more complicated. Recall that the partial order on $S/{\sim_N}$ is defined by~\eqref{eq:order}.  Let us introduce the following notation. For every finite subset $F \ssubset S$ and $[s]\in [F]\ssubset S/{\sim_N}$, denote by $[F_s]\subset [F]$ the subset of elements strictly larger than $[s]$.

Next, given $a,b\in S$ such that $aS\cap bS\ne\emptyset$ and a finite set $[F]\subset\{[s]\in S/{\sim_N}:[a]\vee[b]\le[s]\}$, denote by $T^{a,b}_{[F]}$ the set of all elements $[s]\in [F]$ such that there is $[t]\in S/{\sim_N}$ satisfying
\begin{equation}\label{eq:t conditions}
[s]\le a[t],\ \ [s]\le b[t],\ \ \text{but}\ \ [r]\not\le a[t],\ [r]\not\le b[t]\ \ \text{for all}\ \ [r]\in [F_s].
\end{equation}

Equivalently, this set can be described as follows. Recall from Lemma~\ref{lem:action} that $S$ acts on $S/{\sim_N}$ by injective maps. The assumption that all elements $[s]\in [F]$ satisfy $[s]\ge [a]\vee [b]$ implies that we have well-defined elements $a^{-1}[s]$ and $b^{-1}[s]$. It follows that, for every $[s] \in T^{a,b}_{[F]}$, the smallest $[t]$ with \eqref{eq:t conditions} is given by $a^{-1}[s]\vee b^{-1}[s]$. Therefore  $T^{a,b}_{[F]}$ consists of the elements $[s]\in [F]$ such that
$$
a^{-1}[s]\vee b^{-1}[s]<\infty\ \ \text{and}\ \ a^{-1}[r]\not\le a^{-1}[s]\vee b^{-1}[s],\  b^{-1}[r]\not\le a^{-1}[s]\vee b^{-1}[s] \ \ \text{for all}\ \ [r]\in [F_s].
$$

\begin{lemma}\label{lem:Fab}
Assume $a,b\in S$ are such that $aS\cap bS\ne\emptyset$. Then, modulo a set of measure zero, we have
$$
\Omega^{a,b}_\fix=\bigcap_{[F]}\bigcup_{[s]\in T^{a,b}_{[F]}}Z_{s,F_s},
$$
where the intersection is over all finite  $\vee$-closed subsets $[F]$ of $\{[t]\in S/{\sim_N}:[a]\vee[b]\le[t]\}$ containing $[a]\vee [b]$. Hence
\begin{equation}\label{eq:Fab}
\mu_N(\Omega^{a,b}_\fix)=\lim_{[F]}\sum_{[s]\in T^{a,b}_{[F]}}\left(N(s)^{-1}+\sum_{\emptyset\neq K\subset [F_s]} (-1)^{\vert K\vert} N(q_K)^{-1}\right),
\end{equation}
where the limit is over the finite sets $[F]$ as above ordered by inclusion.
\end{lemma}

Here by a $\vee$-closed set we mean a subset $[F]\subset S/{\sim_N}$ such that $[s]\vee[t]\in[F]$ whenever $[s],[t]\in [F]$ and $[s]\vee[t]<\infty$.

\bp
Let $c$ be such that $aS\cap bS=cS$. First of all observe that the set of points fixed by $\lambda_a^{\phantom{1}}\lambda_b^{-1}$ is contained in the compact open set $\lambda_a\cdot Z_e\cap\lambda_b\cdot Z_e=Z_c$.

Consider now a finite subset $F$ of $cS$ containing $c$ such that the sets $sS$, $s\in F$, are pairwise distinct and the collection of these sets is closed under intersections. For every $s\in F$ denote by $\tilde F_s\subset F$ the set of elements $t\in F$ such that $tS\subsetneq sS$. Then the atoms of the algebra of subsets of~$Z_c$ generated by $Z_s$, $s\in F$, are exactly the sets $Z_{s,\tilde F_s}$. Denote by $\tilde T^{a,b}_F$ the set of $s\in F$ such that $\lambda_a^{-1}\cdot Z_{s,\tilde F_s}\cap \lambda_b^{-1}\cdot Z_{s,\tilde F_s}\ne\emptyset$. It follows that $\Omega^{a,b}_\fix\subset\cup_{s\in \tilde T^{a,b}_F}Z_{s,\tilde F_s}$, hence $\Omega^{a,b}_\fix\subset\cap_F\cup_{s\in \tilde T^{a,b}_F}Z_{s,\tilde F_s}$,
where the intersection is over all finite sets $F$ as above.

Assume now that $\chi$ is an element of $\cap_F\cup_{s\in \tilde T^{a,b}_F}Z_{s,\tilde F_s}$. Then, since the sets $Z_{s,\tilde F_s}$ form a basis of topology on $Z_c$, we conclude that $\lambda_a^{-1}\cdot U\cap \lambda_b^{-1}\cdot U\ne\emptyset$ for every neighbourhood $U$ of $\chi$. Hence $\lambda_a^{-1}\cdot \chi=\lambda_b^{-1}\cdot \chi$, so that $\chi$ is fixed by $\lambda_a^{\phantom{1}}\lambda_b^{-1}$. Therefore
$$
\Omega^{a,b}_\fix=\bigcap_F\bigcup_{s\in \tilde T^{a,b}_F}Z_{s,\tilde F_s}.
$$

In order to prove the lemma it now suffices to show that for every finite set $F\subset cS$ as above we have, modulo a set of measure zero,
\begin{equation}\label{eq:twoT}
\bigcup_{s\in \tilde T^{a,b}_F}Z_{s,\tilde F_s}=C_F\cup\bigcup_{[s]\in T^{a,b}_{[F]}}Z_{s,F_s}
\end{equation}
for some Borel set $C_F$ such that $\mu_N(\Omega^{a,b}_\fix\cap C_F)=0$. Indeed, this would prove the first statement of the lemma, and then an application of~\eqref{eq:muN} would give~\eqref{eq:Fab}.

To establish~\eqref{eq:twoT} it is enough to show the following two properties:
\begin{enumerate}
\item[(i)] for every $[s]\in T^{a,b}_{[F]}$, there is $s'\in \tilde T^{a,b}_F$ such that $[s']=[s]$ and $[\tilde F_{s'}]=[F_s]$;
\item[(ii)] for every $s\in \tilde T^{a,b}_{F}$, we have either $[s]\in T^{a,b}_{[F]}$ and $[\tilde F_s]=[F_s]$, or $\mu_N(\Omega^{a,b}_\fix\cap Z_{s,\tilde F_s})=0$.
\end{enumerate}

We start with (i). Take $[s]\in T^{a,b}_{[F]}$. Since by assumption the collection of the sets $s'S$, with $s'\in F$ and $s'\sim_Ns$, is closed under intersections, it has a smallest element $s'S$. Then, for every $r\in\tilde F_{s'}$, we have $r\not\sim_Ns'$, hence $[r]\in [F_s]$. On the other hand, if $[r]\in[F_s]$, then $s'S\cap rS=r'S$ for some $r'\in F$, $r'\sim_N r$. Hence $r'\in\tilde F_{s'}$ and $[r']=[r]$. Therefore $[\tilde F_{s'}]=[F_s]$.

By assumption, we have $a^{-1}[s]\vee b^{-1}[s]<\infty$ and $a^{-1}[r]\not\le a^{-1}[s]\vee b^{-1}[s]$, $ b^{-1}[r]\not\le a^{-1}[s]\vee b^{-1}[s]$ for all $[r]\in [F_s]$. It follows that there exists $t\in S$ such that $a^{-1}s'S\cap b^{-1}s'S=tS$, and then $t\not\in a^{-1}rS$ and $t\not\in b^{-1}rS$ for all $r\in \tilde F_{s'}$. This shows that
$$
\lambda_a^{-1}\cdot Z_{s,\tilde F_s}\cap \lambda_b^{-1}\cdot Z_{s,\tilde F_s}=Z_{t,a^{-1}\tilde F_s\cup b^{-1}\tilde F_s}\ne\emptyset,
$$
that is, $s'\in \tilde T^{a,b}_F$.

Turning to (ii), take $s\in \tilde T^{a,b}_{F}$. If there is $s'\in\tilde F_s$ equivalent to $s$, then $Z_{s,\tilde F_s}\subset Z_s\setminus Z_{s'}$, hence $Z_{s,\tilde F_s}$ has measure zero. Assume now that there is no such $s'$, or in other words, $sS$ is the smallest set among the sets $s'S$ with $s'\in F$ and $s'\sim_Ns$. As we have already seen in the proof of (i) this implies that $[\tilde F_s]=[F_s]$.

The assumption $\lambda_a^{-1}\cdot Z_{s,\tilde F_s}\cap \lambda_b^{-1}\cdot Z_{s,\tilde F_s}\ne\emptyset$ is equivalent to the existence of $t\in S$ such that $a^{-1}s\cap b^{-1}s=tS$, $t\not\in a^{-1}rS$ and $t\not\in b^{-1}rS$ for all $r\in \tilde F_{s}$. If $a^{-1}[r]\not\le [t]$ and $b^{-1}[r]\not\le [t]$ for all $r\in \tilde F_s$, then we get $[s]\in T^{a,b}_{[F]}$. Assume next that one of these conditions is violated, say, $a^{-1}[r]\le [t]$ for some $r\in \tilde F_s$. Then $\mu_N(Z_{at}\setminus Z_r)=0$, and since
$$
\Omega^{a,b}_\fix\cap Z_{s,\tilde F_s}\subset \lambda_a\cdot(\lambda_a^{-1}\cdot Z_{s,\tilde F_s}\cap \lambda_b^{-1}\cdot Z_{s,\tilde F_s})=
\lambda_a\cdot Z_{t,a^{-1}\tilde F_s\cup b^{-1}\tilde F_s}\subset Z_{at,\tilde F_s}\subset Z_{at}\setminus Z_r,
$$
we see that $\mu_N(\Omega^{a,b}_\fix\cap Z_{s,\tilde F_s})=0$. The case $b^{-1}[r]\le [t]$ for some $r\in \tilde F_s$ is similar.
\ep

Combining these results we get the following criterion.

\begin{thm}\label{thm:unique}
Assume $S$ is a countable right LCM semigroup and $N$ is a scale on $S$ such that there is a probability measure $\mu_N$ on $\widehat{E(S)}$ satisfying $\mu_N(Z_s)=N(s)^{-1}$, $s\in S$, or equivalently, such that there is a $\sigma^N$-KMS$_1$-state on $C^*(S)$.
Then \eqref{eq:2states 1} and \eqref{eq:2states 2} define $\sigma^N$-KMS$_1$-states $\varphi'$ and~$\varphi''$, whose values at $v_a^{\phantom{*}}v_b^*$ are given by the limits~\eqref{eq:Aab} and~\eqref{eq:Fab}, respectively, if  $aS\cap bS\ne\emptyset$ and $N(a)=N(b)$, and $\varphi'(v_a^{\phantom{*}}v_b^*)=\varphi''(v_a^{\phantom{*}}v_b^*)=0$ otherwise. Furthermore, for any $\sigma^N$-KMS$_1$-state~$\varphi$ we then have
$$
|\varphi(v_a^{\phantom{*}}v_b^*)-\varphi'(v_a^{\phantom{*}}v_b^*)|\le \varphi''(v_a^{\phantom{*}}v_b^*)-\varphi'(v_a^{\phantom{*}}v_b^*)\ \ \text{for all}\ \ a,b\in S.
$$

In particular, there is exactly one $\sigma^N$-KMS$_1$-state on $C^*(S)$ if and only if the numbers~\eqref{eq:Aab} and~\eqref{eq:Fab} are equal for all $a,b\in S$ such that $a\ne b$, $aS\cap bS\ne\emptyset$ and $N(a)=N(b)$.
\end{thm}

\bp
As we already observed, the states $\varphi'$ and $\varphi''$ are given by~\eqref{eq:phis}.
Observe also that if $aS\cap bS=\emptyset$, then $\Omega^{a,b}_\fix=\emptyset$, and if $N(a)\ne N(b)$, then $\mu_N(\Omega^{a,b}_\fix)=0$, as~$\lambda_a^{\phantom{1}}\lambda_b^{-1}$ rescales the measure~$\mu_N$ by $N(a)N(b)^{-1}$. Together with Lemmas~\ref{lem:Aab} and~\ref{lem:Fab} this shows that $\varphi'$ and $\varphi''$ have the asserted values at $v_a^{\phantom{*}}v_b^*$.

By Theorem~\ref{thm:KMS} and the explicit form~\eqref{eq:groupoid-iso} of the isomorphism $C^*(S)\cong C^*(\CG(S))$, for any other KMS$_1$-state $\varphi$ we have
$$
\varphi(v_a^{\phantom{*}}v_b^*)=\int_{\Omega^{a,b}_\fix}\tau_\chi(u_{[\lambda_a^{\phantom{1}}\lambda_b^{-1},\chi]})d\mu_N(\chi)
$$
for a measurable field of tracial states $\tau_\chi$ on $C^*(\CG(S)^\chi_\chi)$. It follows that, for all $a,b\in S$,
\begin{align*}
|\varphi(v_a^{\phantom{*}}v_b^*)-\varphi'(v_a^{\phantom{*}}v_b^*)|&=\left|\int_{\Omega^{a,b}_\fix\setminus\Omega^{a,b}_\triv}
\tau_\chi(u_{[\lambda_a^{\phantom{1}}\lambda_b^{-1},\chi]})d\mu_N(\chi)\right|\\
&\le\mu_N(\Omega^{a,b}_\fix\setminus\Omega^{a,b}_\triv)=\varphi''(v_a^{\phantom{*}}v_b^*)-\varphi'(v_a^{\phantom{*}}v_b^*).
\end{align*}

Finally, by the above inequality or by Corollary~\ref{cor:uniqueKMS}, the $\sigma^N$-KMS$_1$-state is unique if and only if $\varphi'(v_a^{\phantom{*}}v_b^*)=\varphi''(v_a^{\phantom{*}}v_b^*)$ for all $a,b\in S$. This condition is of interest only for $a$ and $b$ such that $a\ne b$, $aS\cap bS\ne\emptyset$ and $N(a)=N(b)$.
\ep

Under rather mild additional conditions the uniqueness criterion can be simplified considerably.

\smallskip

First of all we have the following generalization of~\cite{BLRSi}*{Proposition~2.2}, showing that often it suffices to consider a much smaller set of pairs $(a,b)$.

\begin{lemma}\label{lem:ker-reduction}
Under the assumptions of Theorem~\ref{thm:unique}, assume in addition that $1$ is an isolated point of $N(S)\subset[1,+\infty)$. Then any $\sigma^N$-KMS$_1$-state on $C^*(S)$ is completely determined by its restriction to~$C^*(\ker N)$.

In particular, if the numbers \eqref{eq:Aab} and \eqref{eq:Fab} are different for some $a,b\in S$, with $a\ne b$, $aS\cap Sb\ne\emptyset$ and $N(a)=N(b)$, then the same is true already for some $a,b\in\ker N$.
\end{lemma}

\bp
The proof is essentially identical to that of~\cite{BLRSi}*{Proposition~2.2}.
Assume $\varphi$ is a $\sigma^N$-KMS$_1$-state on $C^*(S)$. Take $a,b\in S$. We want to show how to compute $\varphi(v_a^{\phantom{*}}v_b^*)$ only knowing~$\varphi|_{C^*(\ker N)}$.

If $N(a)\ne N(b)$, then $\varphi(v_a^{\phantom{*}}v_b^*)=0$ by $\sigma^N$-invariance of $\varphi$. If $N(a)=N(b)=1$, then there is nothing to show. Therefore we may assume that $N(a)=N(b)>1$.

If $aS\cap bS=\emptyset$, then $v_b^*v_a^{\phantom{*}}=0$, hence
$$
\varphi(v_a^{\phantom{*}}v_b^*)=N(a)^{-1}\varphi(v_b^*v_a^{\phantom{*}})=0.
$$

Assume now that $aS\cap bS\ne\emptyset$. Let $c_1\in S$ be such that $aS\cap bS=c_1S$, and put $b_1:=a^{-1}c_1$, $a_1:=b^{-1}c_1$. Then
$$
\varphi(v_a^{\phantom{*}}v_b^*)=N(a)^{-1}\varphi(v_b^*v_a^{\phantom{*}})=N(a)^{-1}\varphi(v_b^*v_{c_1}^{\phantom{*}}v_{c_1}^*v_a^{\phantom{*}})
=N(a)^{-1}\varphi(v_{a_1}^{\phantom{*}}v_{b_1}^*).
$$

If $a_1,b_1\in\ker N$, then we are done. Otherwise we apply the same argument to the pair $(a_1,b_1)$ instead of~$(a,b)$, and so~on. This process either stops at some point, which means that we succeeded in computing $\varphi(v_a^{\phantom{*}}v_b^*)$ from $\varphi|_{C^*(\ker N)}$, or it continues indefinitely. In the latter case, after $n$ steps, we get elements $a_1,\dots,a_n,b_1,\dots,b_n\in S$ such that $N(a_k)=N(b_k)>1$ and
$$
\varphi(v_a^{\phantom{*}}v_b^*)=N(a)^{-1}N(a_1)^{-1}\dots N(a_{n-1})^{-1}\varphi(v_{a_n}^{\phantom{*}}v_{b_n}^*).
$$
But by our assumption on $N$ we have $N(s)\ge 1+\delta$ for some $\delta>0$ and all $s\in S\setminus\ker N$. Hence
$$
|\varphi(v_a^{\phantom{*}}v_b^*)|\le (1+\delta)^{-n}.
$$
Since this inequality holds for all $n$, we conclude that $\varphi(v_a^{\phantom{*}}v_b^*)=0$.
\ep

Second, the sets $T^{a,b}_{[F]}$ can sometimes be given a better description.

\begin{lemma}\label{lem:Fab-fix}
Assume $a,b\in\ker N$ and $[F]\subset S/{\sim_N}$ is a finite $\vee$-closed set containing $[e]$ and invariant under the actions of $a$ and $b$. Then
$$
T^{a,b}_{[F]}=\{[s]\in [F]: a^{-1}[s]=b^{-1}[s]\}.
$$
\end{lemma}

\bp
Assume $[s]\in T^{a,b}_{[F]}$. Then $[t]:=a^{-1}[s]\vee b^{-1}[s]$ is a well-defined element of $[F]$. By definition we have $[r]\not\le a[t]=[s]\vee ab^{-1}[s]$ for all $[r]\in [F_s]$. But this implies that $ab^{-1}[s]\le [s]$, as otherwise we could take $[r]=[s]\vee ab^{-1}[s]$ and get a contradiction. Thus $b^{-1}[s]\le a^{-1}[s]$. By symmetry we also have the opposite inequality, hence $a^{-1}[s]=b^{-1}[s]$.

Conversely, assume $a^{-1}[s]=b^{-1}[s]$. Then it is obvious that $a^{-1}[r]\not\le a^{-1}[s]\vee b^{-1}[s]=a^{-1}[s]$ and  $b^{-1}[r]\not\le a^{-1}[s]\vee b^{-1}[s]=b^{-1}[s]$ for all $[r]\in[F_s]$, hence $[s]\in T^{a,b}_{[F]}$.
\ep

If the action of $\ker N$ on $S/{\sim_N}$ has finite orbits, then, when we compute~\eqref{eq:Fab}, we can work only with $(\ker N)$-invariant sets $[F]$ and use the above description of $T^{a,b}_{[F]}$. Note that a simple condition which implies finiteness of the orbits is that the sets $N^{-1}(\lambda)/{\sim_N}$, $\lambda\in N(S)$, are finite.

\begin{example}\label{ex:generalized_scale}
Consider the setup from~\cite{BLRS}. Thus, we assume that $N$ is a nontrivial scale that takes values in $\N$ and satisfies the following conditions:
\begin{enumerate}
\item[(1)] $\ker N$ coincides with the core subsemigroup $S_c$ (hence $\sim_N$ is well-defined);
\item[(2)] $|N^{-1}(n)/{\sim_N}|=n$ for all $n\in N(S)$;
\item[(3)] if $N(s)=N(t)$, then either $s\sim_Nt$ or $sS\cap tS=\emptyset$;
\item[(4)] for all $s\in S$ and $n\in N(S)$, there is $t\in S$ such that $N(t)=n$ and $sS\cap tS\ne\emptyset$.
\end{enumerate}
It is shown in \cite{Stam} that for every $S$ there is at most one such scale $N$.

The scale $N$ satisfies condition~\eqref{eq:admis} by Lemma~\ref{lem:finite-level} (see also~\cite{Stam}*{Lemma 2.1} for a different argument). The corresponding $\zeta$-function is
$$
\zeta_N(\beta)=\sum_{n\in N(S)}n^{1-\beta}.
$$
It follows that its abscissa of convergence $\beta_c$ lies in the interval $[1,2]$. Consider three cases.

\smallskip

(i) $\beta>\beta_c$. In this case we are in the setting of Theorem~\ref{thm:finiteKMS}, so the measure $\mu_{N,\beta}$ is well-defined and the $\sigma^N$-KMS$_\beta$-states on $C^*(S)$ are parameterized by the tracial states on $C^*(\ker N)$, which recovers~\cite{BLRS}*{Theorem~3.3(2)}. But as a consistency check we can still try to apply Theorem~\ref{thm:unique} (to the scale $N^\beta$).

By Lemma~\ref{lem:ker-reduction} it suffices to compare $\Omega^{a,b}_\fix$ and $\Omega^{a,b}_\triv$ for $a,b\in\ker N$, $a\ne b$. By the proof of Theorem~\ref{thm:finiteKMS} the measure space $(\widehat{E(S)},\mu_{N,\beta})$ can be identified with $(S/{\sim_N},\zeta_N(\beta)^{-1}\sum_{[s]}N(s)^{-\beta}\delta_{[s]})$. By Lemma~\ref{lem:Aab}, under this identification, the set $\Omega^{a,b}_\triv$ is simply the set
$$
B^{a,b}=\{[s]\in S/{\sim_N}\mid s'=at=bt\ \text{for some}\ s'\sim_Ns\ \text{and}\ t\in S\},
$$
so that
$$
\mu_{N,\beta}(\Omega^{a,b}_\triv)=\zeta_N(\beta)^{-1}\sum_{[s]\in B^{a,b}}N(s)^{-\beta}.
$$
By letting $B^{a,b}_n:=B^{a,b}\cap(N^{-1}(n)/{\sim_N})$, we can write this as
$$
\mu_{N,\beta}(\Omega^{a,b}_\triv)=\zeta_N(\beta)^{-1}\sum_{n\in N(S)}n^{-\beta}|B^{a,b}_n|.
$$

Next, recall that by~\cite{BLRS}*{Proposition~3.6}, if $sS\cap tS=rS$ for some $r,s,t\in S$, then $N(s)N(S)\cap N(t)N(S)=N(r)N(S)$. For every $n\in N(S)$, consider the union $[F_n]$ of the sets $N^{-1}(m)/{\sim_N}$ over all $m\in N(S)$ such that $n\in mN(S)$. The sets $[F_n]$ are finite, $\vee$-closed, and contain $[e]$.
Note also that $[F_m]\subset[F_n]$ if $m$ divides $n$ in $N(S)$.
Hence, by Lemma~\ref{lem:Fab}, modulo a set of measure zero, we have
$$
\Omega^{a,b}_\fix=\bigcap_{n\in N(S)}\bigcup_{[s]\in T^{a,b}_{[F_n]}}Z_{s,(F_n)_s}.
$$
Since the sets $[F_n]$ are invariant under the actions of $a$ and $b$, using Lemma~\ref{lem:Fab-fix} and identifying our measure space with $S/{\sim_N}$ as before, we see that the set on the right hand side above is
$$
T^{a,b}:=\{[s]\in S/{\sim_N}: a^{-1}[s]=b^{-1}[s]\}.
$$
Hence, letting $T^{a,b}_n:=T^{a,b}\cap(N^{-1}(n)/{\sim_N})$, we get
$$
\mu_{N,\beta}(\Omega^{a,b}_\fix)=\zeta_N(\beta)^{-1}\sum_{n\in N(S)}n^{-\beta}|T^{a,b}_n|.
$$

Therefore Theorem~\ref{thm:unique} tells us that there is exactly one $\sigma^N$-KMS$_\beta$-state if and only if $B^{a,b}=T^{a,b}$ for all $a,b\in\ker N$, $a\ne b$. The last condition is easily seen to be equivalent to triviality of $I_\ell(\ker N)/\gamma$, which is consistent with Theorem~\ref{thm:finiteKMS}.

\smallskip

(ii) $1<\beta\le\beta_c$. By \cite{Stam}*{Proposition~3.1}, the semigroup $N(S)$ is the free abelian monoid generated by its irreducible elements. For every finite set $I\subset\Irr(N(S))$, consider $S_I:=N^{-1}(\langle I\rangle)$, where $\langle I\rangle$ is the monoid multiplicatively generated by $I$. The property $N(s)N(S)\cap N(t)N(S)=N(r)N(S)$ for $sS\cap tS=rS$ implies that $S_I$ is a hereditary LCM submonoid of $S$ and $N_I:=N|_{S_I}$ satisfies properties (1)--(4). As the abscissa of convergence of $\zeta_{N_I}$ is $1$, the considerations in (i) apply to $S_I$.

In particular, the measures $\mu_{N_I,\beta}$ are well-defined. Hence $\mu_{N,\beta}$ is well-defined, e.g., by Proposition~\ref{prop:muN}, according to which we need to verify a system of inequalities, but each of them can be considered as an inequality for $S_I$ for some $I$.

Next, Lemmas~\ref{lem:Aab} and~\ref{lem:Fab} show that to compute the measures of $\Omega^{a,b}_\triv$ and $\Omega^{a,b}_\fix$ we need to perform certain computations for finite subsets $[F]\subset S/{\sim_N}$ and pass to the limit. But as each $[F]$ lies in some $S_I/{\sim_N}$, we can write this as a double limit by first letting $[F]$ increase in $S_I/{\sim_N}$ and then letting $I\nearrow\Irr(N(S))$. For the first limit we can use the computations in (i). Thus, for all $a,b\in\ker N$,  we get
\begin{align*}
\mu_{N,\beta}(\Omega^{a,b}_\triv)&=\lim_I\zeta_{N_I}(\beta)^{-1}\sum_{n\in \langle I\rangle }n^{-\beta}|B^{a,b}_n|,\\
\mu_{N,\beta}(\Omega^{a,b}_\fix)&=\lim_I\zeta_{N_I}(\beta)^{-1}\sum_{n\in \langle I\rangle}n^{-\beta}|T^{a,b}_n|.
\end{align*}

The conclusion is that there is a $\sigma^N$-KMS$_\beta$-state on $C^*(S)$, and such a state is unique if and only if the limits above are equal for all $a,b\in\ker N$, $a\ne b$. This recovers \cite{BLRS}*{Theorem 3.3(5)} and shows that the sufficient condition for the uniqueness in~\cite{BLRS} is also necessary. Indeed, the sets~$A^{a,b}_n$ and~$F^{a,b}_n$ from~\cite{BLRS} are related to our sets $B^{a,b}_n$ and $T^{a,b}_n$ by $B^{a,b}_n=aA^{a,b}_n=bA^{a,b}_n$, $T^{a,b}_n=aF^{a,b}_n=bF^{a,b}_n$.

\smallskip

(iii) $\beta=1$. In this case the measure $\mu_N$ exists and is the limit of $\mu_{N,\beta}$ as $\beta\downarrow1$. The computations for this measure get easier thanks to the property that for every $n\in N(S)$ the sets $Z_s$, $[s]\in N^{-1}(n)/{\sim_N}$, are disjoint and their union is a set of full measure. This implies that if $[F_n]$ are the sets introduced in (i), then for every $[s]\in [F_n]$ we have
$$
Z_s=\bigsqcup_{[t]\in N^{-1}(n)/{\sim_N}:[s]\le[t]}Z_t,
$$
modulo a set of measure zero. Since the set $B^{a,b}$ is directed, in the sense that if $[s]\in B^{a,b}$ and $[s]\le[t]$ then also $[t]\in B^{a,b}$, we conclude that, modulo a set of measure zero,
$$
\bigcup_{[s]\in B^{a,b}\cap [F_n]} Z_s=\bigsqcup_{[s]\in B^{a,b}_n}Z_s.
$$
Hence, by Lemma~\ref{lem:Aab}, for all $a,b\in\ker N$, we have
$$
\mu_{N}(\Omega^{a,b}_\triv)=\lim_{n\in N(S)}\mu_N\bigg(\bigcup_{[s]\in B^{a,b}\cap [F_n]} Z_s\bigg)=\lim_{n\in N(S)}\frac{|B^{a,b}_n|}{n},
$$
where the limits are taken over the monoid $N(S)$ ordered by divisibility.

Similar considerations apply to $\Omega^{a,b}_\fix$: we have $\mu_N(Z_{s,(F_n)_s})=0$ for $[s]\in F_n$, $N(s)\ne n$, so that, modulo a set of measure zero,
$$
\bigcup_{[s]\in T^{a,b}_{[F_n]}}Z_{s,(F_n)_s}=\bigsqcup_{[s]\in T^{a,b}_n}Z_s.
$$
Hence, by Lemma~\ref{lem:Fab},
$$
\mu_{N}(\Omega^{a,b}_\fix)=\lim_{n\in N(S)}\frac{|T^{a,b}_n|}{n}.
$$

Therefore we conclude that there is a $\sigma^N$-KMS$_1$-state on $C^*(S)$, and such a state is unique if and only if the limits above are equal for all $a,b\in\ker N$, $a\ne b$. This recovers \cite{BLRS}*{Theorem 3.3(4)} and shows that the sufficient condition for the uniqueness in~\cite{BLRS} is also necessary.
\hfill$\diamondsuit$
\end{example}

\begin{example}
Consider the right LCM semigroup $S=G\rtimes\Z^2_+$ from~\cite{Stam}*{Example~5.3}, where $G=\bigoplus_{\Z_+}\Z/2\Z$ and $(x,y)\in\Z^2_+$ acts on $G$ by the injective endomorphism $\sigma^x(\operatorname{id}+\sigma)^y$, where $\sigma\colon G\to G$ is the shift: $\sigma(g_0,g_1,\dots)=(0,g_0,g_1,\dots)$. Define a scale $N$ on $S$ by
$N(g,x,y)=2^{x+y}$. This scale is not a generalized scale in the sense of the previous example, moreover, the semigroup~$S$ does not admit a generalized scale~\cite{Stam}. Nevertheless, as we will see shortly, the analysis of the pair $(S,N)$ is very similar to the previous example. More precisely, the picture is almost identical to that for the right LCM semigroup $\Z\rtimes P_{p,q}$ with a generalized scale, where $P_{p,q}\subset\N$ is the submonoid multiplicatively generated by two different primes $p$ and $q$.

We have $\ker N=G$. Denote by $G_{x,y}$ the image of $G$ under $\sigma^x(\operatorname{id}+\sigma)^y$. Then $S/{\sim_N}$ can be identified with the disjoint union $\bigsqcup_{(x,y)\in\Z^2_+}G/G_{x,y}$, so that the class of $(g,x,y)\in S$ in $S/{\sim_N}$ is the class $[g]\in G/G_{x,y}$, cf.~Example~\ref{ex:ax+b}. The partial order on $S/{\sim_N}$ is described as follows: given $[g]\in G/G_{x,y}$ and $[h]\in G/G_{u,v}$, we have $[g]\le[h]$ if and only if
$$
x\le u,\ \ y\le v,\ \ g\equiv h\mod G_{x,y}.
$$

As $[G:G_{0,1}]=[G:G_{1,0}]=2$ and the homomorphisms $\sigma$ and $\operatorname{id}+\sigma$ are injective, we have
$$
[G:G_{x,y}]=[G:G_{x,0}]\,[G_{x,0}:G_{x,y}]=[G:G_{x,0}]\,[G:G_{0,y}]=2^{x+y}.
$$
Hence the $\zeta$-function of $N$ equals
$$
\zeta_N(\beta)=\sum^\infty_{x,y=0}|G/G_{x,y}|\,2^{-\beta(x+y)}=\sum^\infty_{x,y=0}2^{-(\beta-1)(x+y)}=(1-2^{-(\beta-1)})^{-2}.
$$
Therefore $\beta_c=1$. Thus, by Theorem~\ref{thm:finiteKMS}, for $\beta>1$ we have an affine bijective correspondence between the $\sigma^N$-KMS$_\beta$-states on $C^*(S)$ and the states on the commutative C$^*$-algebra $C^*(G)$, or in other words, the probability measures on $\{0,1\}^{\Z_+}$.

Next, we claim that for $\beta<1$ there are no KMS$_\beta$-states. Indeed, assume that the measure $\mu_{N,\beta}$ exists for some $\beta$. By our description of the order structure on $S/{\sim_N}$, if $s,t\in S$ are such that their equivalence classes are equal to two different elements of $G/G_{1,0}$, then $Z_s\cap Z_t=\emptyset$. Hence
$$
1\ge\mu_{N,\beta}(Z_s)+\mu_{N,\beta}(Z_t)=2\cdot 2^{-\beta},
$$
which forces $\beta\ge1$. Note for a later use that for $\beta=1$ the measure $\mu_N$ exists, since the measures~$\mu_{N,\beta}$ exist for all $\beta>1$, and the above argument shows that
\begin{equation}\label{eq:zero}
\mu_N(Z_{e,\{s,t\}})=0.
\end{equation}

It remains to understand the case $\beta=1$. Let us try to check whether the uniqueness criterion from Theorem~\ref{thm:unique} is satisfied. By Lemma~\ref{lem:ker-reduction} it suffices to compare $\Omega^{a,b}_\fix$ and $\Omega^{a,b}_\triv$ for $a,b\in G$, $a\ne b$. Since $S$ is right cancellative, we have $\Omega^{a,b}_\triv=\emptyset$, so we just need to check whether $\mu_N(\Omega^{a,b}_\fix)=0$.

For this we have to understand the operation $\vee$ on $S/{\sim_N}$. We start by observing that
\begin{equation}\label{eq:cap}
G_{x,y}=G_{x,0}\cap G_{0,y}.
\end{equation}
Indeed, writing $\operatorname{id}$ as $(-\sigma+(\operatorname{id}+\sigma))^{x+y}$, we first of all can conclude that
\begin{equation}\label{eq:cup}
G=G_{x,0}+G_{0,y}.
\end{equation}
But then
$$
[G_{x,0}: G_{x,0}\cap G_{0,y}]=[G:G_{0,y}]=[G_{x,0}:G_{x,y}],
$$
and as $G_{x,y}\subset G_{x,0}\cap G_{0,y}$, we get $G_{x,y}=G_{x,0}\cap G_{0,y}$. From \eqref{eq:cap} and \eqref{eq:cup} we deduce by applying powers of $\sigma$ and $\operatorname{id}+\sigma$ that, more generally,
$$
G_{x,y}\cap G_{u,v}=G_{x\vee u,y\vee v}\ \ \text{and}\ \ G_{x,y}+G_{u,v}=G_{x\wedge u,y\wedge v},
$$
where $\vee$ and $\wedge$ are the operations $\max$ and $\min$ on $\Z_+$.

Now, take $[g]\in G/G_{x,y}$ and $[h]\in G/G_{u,v}$. If these elements are dominated by $[f]\in G/G_{p,q}$, then $x\vee u\le p$, $y\vee v\le q$ and $g-f\in G_{x,y}$, $h-f\in G_{u,v}$. It follows that $g-h\in G_{x,y}+G_{u,v}=G_{x\wedge u,y\wedge v}$. Conversely, if $g-h\in G_{x\wedge u,y\wedge v}$, then we have $g-g'=h-h'$ for some $g'\in G_{x,y}$ and $h'\in G_{u,v}$. Since $G_{x,y}\cap G_{u,v}=G_{x\vee u,y\vee v}$, the class of the element $f:=g-g'=h-h'$ in $G/G_{x\vee u,y\vee v}$ is independent of the choice of $g'$ and $h'$. It follows that $[g]\vee[h]$ equals this class. To summarize, we have $[g]\vee[h]=\infty$ if $g-h\not\in G_{x\wedge u,y\wedge v}$, and $[g]\vee[h]\in G/G_{x\vee u,y\vee v}$ otherwise.

Returning to the computation of $\mu_N(\Omega^{a,b}_\fix)$, consider the sets
$$
[F_n]:=\bigcup^n_{x,y=0}G/G_{x,y}\subset S/{\sim_N}.
$$
These sets are finite, $\vee$-closed and invariant under the action of $G$. They form an increasing sequence with union~$S/{\sim_N}$.
Hence, by Lemma~\ref{lem:Fab}, we have
$$
\mu_N(\Omega^{a,b}_\fix)=\lim_{n\to\infty}\mu_N\bigg(\bigcup_{[s]\in T^{a,b}_{[F_n]}}Z_{s,(F_n)_s}\bigg).
$$

We claim that only elements of $(G/G_{n,n})\cap T^{a,b}_{[F_n]}$ can give a nontrivial contribution to the measure of $\bigcup_{[s]\in T^{a,b}_{[F_n]}}Z_{s,(F_n)_s}$. Indeed, assume $s\in F_n$ has class $[g]\in G/G_{x,y}$ for some $x\le n$ and $y\le n$ such that at least one inequality is strict. Let us assume $x<n$. Then there are two different elements $[h],[f]\in G/G_{x+1,y}$ such that $h$ and $f$ are equal to $g$ modulo $G_{x,y}$. Lifting them to elements $t,r\in (F_n)_s$ we get, similarly to~\eqref{eq:zero}, that
$$
\mu_N(Z_{s,(F_n)_s})\le\mu_N(Z_{s,\{t,r\}})=2^{-x-y}-2\cdot 2^{-x-y-1}=0.
$$
For $y<n$ the argument is similar. Thus our claim is proved.

By Lemma~\ref{lem:Fab-fix} we also have that $T^{a,b}_{[F_n]}$ consists of the elements of $[F_n]$ on which the actions of~$-a$ and~$-b$ coincide, that is, $T^{a,b}_{[F_n]}$ equals the union of $G/G_{x,y}$ over $x,y\le n$ such that $a-b\in G_{x,y}$. Since only elements of $(G/G_{n,n})\cap T^{a,b}_{[F_n]}$ can give a nontrivial contribution to the measure of $\bigcup_{[s]\in T^{a,b}_{[F_n]}}Z_{s,(F_n)_s}$, we conclude that
$$
\mu_N\bigg(\bigcup_{[s]\in T^{a,b}_{[F_n]}}Z_{s,(F_n)_s}\bigg)=0\ \ \text{if}\ \ a-b\not\in G_{n,n}.
$$
But as $G_{n,n}\subset\sigma^n(G)$ and $a-b\ne0$, we have $a-b\not\in G_{n,n}$ for all $n$ large enough. Therefore $\mu_N(\Omega^{a,b}_\fix)=0$. Hence there is exactly one $\sigma^N$-KMS$_1$-state on~$C^*(S)$.
\hfill$\diamondsuit$
\end{example}

We finish with a short comment on the relation between uniqueness of KMS-states and simplicity of the \emph{boundary quotient} of $C^*(S)$.

Recall that a finite subset $F\subset S$ is called a foundation set if for every $s\in S$ there is $t\in F$ such that $sS\cap tS\ne\emptyset$.
The boundary quotient of $C^*(S)$ is defined as the quotient of $C^*(S)$ by the ideal generated by the elements $\prod_{s\in F}(1-v_s^{\phantom{*}}v_s^*)$, where $F$ runs over the foundation sets of $S$~\cite{BRRW}. In~\cite{Star} it is shown that this quotient is defined by a closed invariant subset $\widehat E_{\mathrm{tight}}(S)$ of $\widehat{E(S)}$.

\smallskip

We are not aware of sufficiently general conditions that guarantee that a KMS-state factors through the boundary quotient, but in principle this is again subject to a verifiable condition.

\begin{lemma}
Assume $S$ is a right LCM semigroup, $N$ is a scale on $S$, and $\varphi$ is a $\sigma^N$-KMS$_1$-state on $C^*(S)$. Then it factors through the boundary quotient if and only if
$$
1+\sum_{\emptyset\neq K\subset [F]} (-1)^{\vert K\vert} N(q_K)^{-1}=0
$$
for all finite subsets $[F]\subset S/{\sim_N}$ such that for every $[s]\in S/{\sim_N}$ there is $[t]\in[F]$ with $[s]\vee[t]<\infty$.
\end{lemma}

\bp
First of all observe that $F\ssubset S$ is a foundation set if and only if $[F]\ssubset S/{\sim_N}$ has the property that for every $[s]\in S/{\sim_N}$ there is $[t]\in[F]$ such that $[s]\vee[t]<\infty$. Hence, by \eqref{eq:muN}, the condition in the formulation of the lemma means precisely that $\varphi$ vanishes on $\prod_{s\in F}(1-v_s^{\phantom{*}}v_s^*)$ for all foundation sets $F$. Since the KMS-condition implies that in order to check that $\varphi$ vanishes on an ideal it suffices to check that it vanishes on positive generators of the ideal, this gives the result.
\ep

For the KMS-states that do factor through the boundary quotient we have the following result.

\begin{proposition}\label{prop:BQ simple}
Assume $S$ is a countable right LCM semigroup, $N$ is a scale on $S$ and $\varphi$ is a $\sigma^N$-KMS$_1$-state on $C^*(S)$. Assume also that the following conditions are satisfied:
\begin{enumerate}
\item $\varphi$ is a unique $\sigma^N$-KMS$_1$-state;
\item $\varphi$ factors through the boundary quotient of $C^*(S)$;
\item the groupoid $\CG_{\mathrm{tight}}(S):=\CG(S)_{\widehat E_{\mathrm{tight}}(S)}$ is Hausdorff (see~\cite{Star}*{Proposition~4.1}) and amenable.
\end{enumerate}
Then the boundary quotient of $C^*(S)$ is simple.
\end{proposition}

\bp
The assumption that $\varphi$ factors through the boundary quotient $C^*(\CG_{\mathrm{tight}}(S))$ implies that the measure $\mu_N$ is concentrated on the closed invariant set $\widehat E_{\mathrm{tight}}(S)\subset\widehat{E(S)}$. By~\cite{Star}*{Lemma~4.2}, the action of $\CG(S)$ on this set is minimal. At the same time the uniqueness of $\varphi$ implies that the set of points in $\widehat{E(S)}$ with trivial isotropy is a set of full measure. By the minimality of the action, it follows that the set of such points in $\widehat E_{\mathrm{tight}}(S)$ is dense, that is, the groupoid $\CG_{\mathrm{tight}}(S)$ is topologically principal. By~\cite{Star}*{Theorem~4.12} we conclude that the boundary quotient is simple.
\ep

\bigskip


\end{document}